\documentclass{amsart}
\usepackage{amsmath, amsthm, amssymb}
\usepackage{amsmath, amssymb}
 \usepackage{amsmath,amscd,amsthm}
\usepackage{mathrsfs}
\usepackage{tikz}
\input xypic
\xyoption{all}

\theoremstyle{plain} 
\newtheorem{theorem}[subsection]{Theorem}
\newtheorem{proposition}[subsection]{Proposition}
\newtheorem{lemma}[subsection]{Lemma}

\theoremstyle{definition}

\newtheorem{example}[subsection]{Example}

\theoremstyle{remark}
\newtheorem{remark}[subsection]{Remark}
\newtheorem*{remark*}{Remark}

\numberwithin{equation}{subsection}

\begin{document}

\title[Biclosed sets, quasitrivial semigroups and oriented matroid]{Biclosed sets, quasitrivial semigroups and oriented matroid}

\author{Weijia Wang}
\address{School of Statistics and Mathematics \\Shanghai Lixin University of Accounting And Finance\\ No. 995 Shangchuan Road, Shanghai, 201209\\ China }
\email{wwang8@alumni.nd.edu}
\author{Rui Wang}
\address{School of Statistics and Mathematics \\Shanghai Lixin University of Accounting And Finance\\ No. 995 Shangchuan Road, Shanghai, 201209\\ China }
\email{18709483691@163.com}
\maketitle

\begin{abstract}
In this paper, we establish a one-to-one correspondence between the set of biclosed sets in an irreducible root system of type $A_n$ and the set of quasitrivial semigroup structures on a set with $n+1$ elements. Building on this correspondence, 
we first generalize this bijection to provide a semigroup structural characterization of the biclosed sets in a standard parabolic subset. In particular, this allows us to derive an enumeration result for the elements in a parabolic weak order of type $A$.
Secondly, we define an index for an arbitrary subset of the root system of type $A_n$, which quantifies their deviation from from being biclosed, and prove that such an index coincides with the associativity index of the associated quasitrivial magma.
Thirdly, we define type $B_n$  quasitrivial semigroups, and prove that they are in bijective with   biclosed sets in a type $B_n$ root system. Finally, by identifying certain biclosed sets with total preorders, we present  a purely combinatorial proof that a root system of type $A$  possesses an oriented matroid structure.
\end{abstract}

\section{Introduction}

In this paper we identify two algebraic-combinatorics objects.

Root systems play an important role in the study of Lie groups, Lie algebras, algebraic groups, Coxeter groups and representation theory. Biclosed sets in a root system have been extensively studied due to their relationship with weak Bruhat order, reflection orders and Hecke algebra, etc.  For a finite Coxeter group, the biclosed sets in a positive system are in bijection with the group elements, and the poset of these biclosed sets, ordered by containment, is isomorphic to the weak Bruhat order of that Coxeter group. In this work, we focus on the biclosed sets in the whole finite root system. These sets are called invertible sets in \cite{biclosedphi}. The understanding of the  biclosed sets in a finite crystallographic root  system is crucial for studying the reflection orders on affine Weyl groups, $\mathfrak{h}$-complementable Lie subalgebras of simple Lie algebras (see \cite{DyerReflOrder}) and complex homogeneous spaces (see \cite{Malychev}).

A semigroup is defined as a set with an associative binary operation. It is a fundamental algebraic structure with broad applications in various fields, particularly in theoretic computer science. A quasitrivial semigroup is a semigroup whose binary operation outputs one of its input. This class of semigroups has generated much attention and has been  well characterized and enumerated in the works of \cite{quusicsemi1} and \cite{quusicsemi2}.

In Section \ref{mainbijA}, we establish a bijection between the set of biclosed sets in the root system of type $A_n$ (which is precisely the root system of the symmetric group $S_{n+1}$ as well as the Lie algebra $\mathfrak{sl}(n+1,\mathbb{C})$) and the set of quasitrivial semigroup structures on a set with $n+1$ elements.  Both sets admit natural $S_{n+1}$-actions and we show that our bijection is $S_{n+1}$-equivariant. Furthermore, we  determine the correspondence between subclasses of biclosed sets and quasitrivial semigroup structures. This bijection enables us to  characterize the associativity of a quasitrivial binary operation in terms of the biclosedness of a suitably constructed set in the root system. Additionally, we transfer the inclusion order on the set of biclosed sets to the set of quasitrivial semigroup structures. We give a characterization of this order without referring to the biclosed sets. Our results can also be interpreted as a categorification of the biclosed sets.

Building on such an identification, we generalize the bijection in several directions.

In Section \ref{parabolicsec}, we characterize the biclosed sets in a (standard) parabolic subset and show that they are in bijection with certain biclosed sets in the whole root system. This enables us to describe them by quasitrivial semigroup structures. Furthermore, we derive a formula for counting the biclosed sets in a (standard) parabolic subset in the root system of type $A$. An understanding of  the biclosed sets in a parabolic set is a fundamental first step toward understanding the parabolic weak order and addressing a related conjecture proposed in \cite{DyerWeakOrder}. 

In Section \ref{indexsec}, we investigate the associativity index of a quasitrivial magma. Such magmas correspond to  subsets of the root system of type $A$. We introduce the concept of the index of a subset of the root system,  which quantifies the deviation of the subset from being biclosed. We prove that the associativity index of a quasitrivial magma coincides with the index of its corresponding subset of the root system.

In Section \ref{bnsec}, we define   quasitrivial semigroups of type $B_n$ which are shown to be in bijection with the biclosed sets in the  root system of type $B_n$. 

In Section \ref{omsec}, we utilize the bijection between biclosed sets and quasitrivial semigroup structures to describe the covectors of a root system of type $A$. This approach enables us to obtain a purely combinatorial proof of the fact that the root system of type $A$ admits a natural oriented matroid structure without using the root system's linear realization.

\section{Background}

\subsection{Quasitrivial semigroup} A \emph{magma} is defined as a set \( X \) endowed with a binary operation \( F: X \times X \rightarrow X \). A \emph{semigroup} is a magma in which the binary operation \( F \) satisfies the associative property, that is,

\[
F(F(a, b), c) = F(a, F(b, c)) \quad \text{for all} \quad a, b, c \in X.
\]

A semigroup \((X, F)\) is termed \emph{commutative} if the operation \( F \) satisfies

\[
F(a, b) = F(b, a) \quad \text{for all} \quad a, b \in X.
\]

A semigroup \((X, F)\) is said to be \emph{anticommutative} if the equality \( F(a, b) = F(b, a) \) implies that \( a = b \).

A semigroup (or magma) \((X, F)\) is called \emph{quasitrivial} if for all \( a, b \in X \), the operation \( F \) yields

\[
F(a, b) \in \{a, b\}.
\]

An \emph{identity element} of a semigroup \((X, F)\), denoted by \( 1 \), is an element that satisfies

\[
F(1, x) = F(x, 1) = x \quad \text{for all} \quad x \in X.
\]

Similarly, a \emph{zero element} of a semigroup \((X, F)\), denoted by \( 0 \), is an element that fulfills

\[
F(0, x) = F(x, 0) = 0 \quad \text{for all} \quad x \in X.
\]

The \emph{projection map} \( \pi_{i}: X \times X \rightarrow X \), for \( i = 1, 2 \), is defined by

\[
\pi_i(a_1, a_2) = a_i.
\]

It is straightforward to verify that \((X, \pi_i)\) constitutes a quasitrivial semigroup.

A \emph{total preorder} $\precsim$ on a set $X$ is defined to be a binary relation on $X$ that satisfies the total and transitive properties. If $x\precsim y$ and $y\precsim x$, we write $x\sim y$ and this is an equivalence relation. Equivalently a total preorder $\precsim$ on $X$ can be viewed as inducing a total order on the set of equivalence classes of $\sim$. From this point onward, we will always adopt this perspective. The notation $x\precsim y$ signifies  that either $x\precnsim y$ or $x\sim y$. In many cases, $x\precnsim y$  is simply abbreviated as  $x\prec y$. 

It is known that a quasitrivial semigroup structure on a set $X$ is equivalent to a total preorder $\precsim$ on $X$ together with $F:X\times X\rightarrow X$ such that
\begin{equation}
  F =
    \begin{cases}
      \max_{\precsim} & \text{if $x\nsim y$}\\
      \pi_1 \, \text{or} \, \pi_2 & \text{when $F$ is restricted on an equivalence class.}
    \end{cases}
\end{equation}
A proof of this result can be found in \cite{quusicsemi1}.

Suppose that $X=X_n:=\{1,2,\cdots,n\}$. One can define a $W:=S_n$ action on the quasitrivial semigroup structures on $X_n$ as follows. For $\sigma\in W, F:X_n\times X_n\rightarrow X_n$
associative and quasitrivial, define $\sigma(F)(a,b)=\sigma(F(\sigma^{-1}(a),\sigma^{-1}(b))).$ We refer the readers to  \cite{quusicsemi1} and \cite{quusicsemi2} for more information on quasitrivial semigroups.

\subsection{The group $S_n^B$}

Denote by $S_n^B$ the group of permutations $\sigma$ of the set $$\{\pm1,\pm2,\cdots,\pm n\}$$ satisfying $\sigma(-k)=-\sigma(k).$
One can write an element $\sigma\in S_n^B$ in the form $$(-a_n,-a_{n-1},\cdots,-a_1,a_1,a_2,\cdots,a_n)$$ where
$a_i=\sigma(i)$ for $1\leq i\leq n.$ The multiplication of $S_n^B$ is defined as $(\sigma_1\sigma_2)(i)=\sigma_1(\sigma_2(i)).$

\subsection{Root systems of type $A_n, B_n$ and its biclosed sets}
Let $W$ be the symmetry group $S_{n+1}$ on $n+1$ elements or $S_n^B$. For $W=S_{n+1}$, the involutions $(i,i+1), 1\leq i\leq n$ are called the Coxeter generators of $W$. Denote $(i,i+1)$ by $s_i$. 
For $W=S_{n}^B$, the involutions $s_{i+1}:=(i-1,i)(-i,-i+1),2\leq i\leq n$ and $s_1=(-1,1)$ are called the Coxeter generators of $W$.

The irreducible \emph{root system} of type $A_n$ (resp. $B_n$) can be constructed in the following way. Let $V$ be an $n$ dimensional $\mathbb{R}-$vector space with a basis $\{\alpha_i,1\leq i\leq n\}$ whose elements are in bijection with the Coxeter generators $s_i, 1\leq i\leq n$.

For type $A_n$, one equips $V$ with the symmetric bilinear form by requiring
$$(\alpha_i,\alpha_j)=0,\, \text{if $|i-j|\geq 2$},$$
$$(\alpha_i,\alpha_j)=-\frac12,\, \text{if $|i-j|=1$},$$
$$(\alpha_i,\alpha_i)=1$$
and extending $(,)$ bilinearly.

For type $B_n$, one equips $V$ with the symmetric bilinear form by requiring
$$(\alpha_i,\alpha_j)=0,\, \text{if $|i-j|\geq 2$},$$
$$(\alpha_i,\alpha_{i+1})=-\frac12,\, \text{for $2\leq i\leq n-1$},$$
$$(\alpha_1,\alpha_2)=-1,$$
$$(\alpha_i,\alpha_i)=2, \text{for $2\leq i\leq n$},$$
$$(\alpha_1,\alpha_1)=1$$
and extending $(,)$ bilinearly.

For a root system of type $A_n$ (resp. $B_n$), map each Coxeter generator of $W=S_{n+1}$ (resp. $S_n^B$) to a linear map on $V$ defined as:
$$s_i\mapsto \phi(s_i):V\rightarrow V, v\mapsto v-2\frac{(v,\alpha_i)}{(\alpha_i,\alpha_i)}\alpha_i.$$

It can be shown that $\phi$ can be extended to a linear representation of $W$. We write $w(v)$ for $\phi(w)(v).$
Then the root system $\Phi$  is $\{w(\alpha_i):\,w\in W, 1\leq i\leq n\}.$ For the type $A_n$, it can be proved that $$\Phi=\{\pm\sum_{j=i}^{k}\alpha_j:\,1\leq i,k\leq n\}.$$
For the type $B_n$, it can be proved that
$$\Phi=\{\pm\sum_{j=i}^{k}\alpha_j:\,1\leq i,k\leq n\}\cup\{\pm(\sum_{j=1}^{k}2\alpha_j+\sum_{l={k+1}}^p\alpha_l):\,1\leq k<n, p\leq n\}$$
An element in $\Phi$ is called a \emph{root}.
A vector space total order on $V$ is a total order $\leq$ defined on $V$ with the following additional properties:

(1) For $v_1,v_2,v_3\in V$, if $v_1<v_2$ then $v_1+v_3<v_2+v_3;$

(2) For $v_1,v_2\in V$ and $0<k\in \mathbb{R}$ (resp. $0>k\in \mathbb{R}$), $v_1<v_2$ implies $kv_1<kv_2$ (resp. $kv_1>kv_2$).

A \emph{positive system} $\Psi^+$ of $\Phi$ is the subset of $\Phi$ consisting of elements greater than $0$ under a vector space total order on $V$. It follows that $\Phi$ is the disjoint union of $\Psi^+$ and $-\Psi^+$. The subset $-\Psi^+$ is called a \emph{negative system} and is denoted by $\Psi^-$.
 A root in $\Psi^+$ (resp. $\Psi^-$) is called a \emph{positive root} (resp. \emph{negative root}).
 A subset $\Delta$ of $\Psi^+$ is called the \emph{simple system} of $\Psi^+$ if $\Delta$ is a basis for $V$ and every positive root is a positive linear combination of the elements in $\Delta$. An element in $\Delta$ is called a \emph{simple root}.

On can check that for type $A_n$ root system, $\Phi^+:=\{\sum_{j=i}^{k}\alpha_j:\,1\leq i,k\leq n\}$ is a positive system and that for type $B_n$ root system, $$\Phi^+:=\{\sum_{j=i}^{k}\alpha_j:\,1\leq i,k\leq n\}\cup\{\sum_{j=1}^{k}2\alpha_j+\sum_{l={k+1}}^p\alpha_l:\,1\leq k<n, p\leq n\}$$ is a positive system. Its simple system $\Delta$ is the set $\{\alpha_i,1\leq i\leq n\}$.
It can be shown that the set of positive systems is $\{w\Phi^+|,w\in W\}$. The positive system $\Phi^+$ will be called the standard positive system.

A subset $C$ of $\Phi$ is \emph{closed} if for $\alpha,\beta\in C,\alpha+\beta\in \Phi$ implies $\alpha+\beta\in C$. For $D\subset \Phi$, a \emph{biclosed} set $B$ in $D$ is a closed set $B$ whose complement in $D$ is also closed. A positive system $\Psi^+$ can always be written in the form $B\cup -(\Phi^+\backslash B)$ where $B$ is a biclosed set in $\Phi^+.$ Every biclosed set in $\Phi^+$ is of the form $\{\alpha\in \Phi^+|w(\alpha)\in \Phi^-\}$ for some unique $w\in W.$ Thus the biclosed sets in $\Phi^+$ are in bijection with the elements in $W$. In this paper we are mainly concerned with the biclosed sets in the whole root system $\Phi$. We denote the set of biclosed sets in $\Phi$ by $\mathcal{B}(\Phi)$.

It is shown in \cite{biclosedphi}  that the biclosed sets in $\Phi$ are of the form
$$(\Psi^+\backslash(\mathbb{R}-\text{span}(\Delta_1)))\cup((\mathbb{R}-\text{span}(\Delta_2))\cap \Phi)$$
where $\Psi^+$ is a positive system and $\Delta_1,\Delta_2$ are two orthogonal subsets of the simple system of $\Psi^+$  (orthogonality here  means that for any $\alpha\in \Delta_1$ and $\beta\in \Delta_2$ one has $(\alpha,\beta)=0$). Following \cite{DyerReflOrder}, we denote such a set by $\Psi^+_{\Delta_1,\Delta_2}$.
It can be shown that any biclosed set $\Psi^+_{\Delta_1,\Delta_2}$ is equal to $w(\Phi^+_{\Delta_1',\Delta_2'})$ for some $w\in W$ and such a representation is unique.
Let $\Psi^+$ be a positive system and $\Delta$ be its simple system. Assume that $\Delta'\subset \Delta$. Denote by $\Phi_{\Delta'}$ the set $\mathbb{R}-\text{span}(\Delta')\cap \Phi$. This is a root subsystem. A biclosed set of the form $w(\Phi^+_{\emptyset,\Delta_2'})$ is called a \emph{parabolic subset} of $\Phi$. A set of the form $\Phi^+_{\emptyset,\Delta_2'}$ is called a \emph{standard parabolic subset}.  A biclosed set of the form $w(\Phi^+_{\Delta_1',\emptyset})$ is called a \emph{horocyclic subset} of $\Phi$.
The stabilizer of a biclosed set $\Psi^+_{\Delta_1,\Delta_2}=w(\Phi^+_{\Delta_1',\Delta_2'})$ is $wUw^{-1}$ where $U$ is the standard parabolic subgroup generated by $s_{\alpha_i},\alpha_i\in \Delta_1'\cup\Delta_2'$.
A detailed study of the poset of the biclosed sets in $\Phi$ under inclusion can  be found in \cite{DyerReflOrder}. For more information on the root system, see \cite{Hum} and \cite{bjornerbrenti}. A combinatorial description of biclosed sets in root systems of classical types (though not their semigroup structural correspondence) is also mentioned in \cite{combinbiclosed}.

\section{The bijection}\label{mainbijA}

Let $n$ be a positive integer.  Let $(k_1,k_2,\cdots, k_t)$ be a vector consisting of $t$ positive integers.
Consider the set
$$P_{(k_1,k_2,\cdots, k_t)}=$$
$$\{(i_1,i_2,\cdots,i_t):1\leq i_1, i_t+k_t-1\leq n, i_j+k_j+1\leq i_{j+1}, 1\leq j\leq t-1\}.$$

For $P_{(k_1,k_2,\cdots, k_t)}\neq \emptyset$ we define $p:=n+1-\sum_{l=1}^tk_l-t$. One easily sees that $p\geq 0$ if $P_{(k_1,k_2,\cdots, k_t)}\neq \emptyset$.
Now we consider two finite sets: $\{A_1,A_2,\cdots, A_t\}$ and $\{B_1,B_2,\cdots,B_p\}$. Define
$$Q_{t,p}=$$
$$\{\text{total}\,\text{order}\,<\,\text{on}\, \{A_1,A_2,\cdots, A_t,B_1,B_2,\cdots,B_p\}: A_1<A_2<\cdots<A_t,$$ $$B_1<B_2<\cdots<B_p\}.$$

\begin{lemma}\label{bijection1}
There exists a bijection between $P_{(k_1,k_2,\cdots, k_t)}$ and $Q_{t,p}$.
\end{lemma}

\begin{proof}
Let $(i_1,i_2,\cdots,i_t)$ be an element in $P_{(k_1,k_2,\cdots, k_t)}$. We define a total order on $\{A_1,A_2,\cdots, A_t, B_1,B_2,\cdots,B_p\}$ as follows.
$$B_1<B_2<\cdots <B_{i_1-1}<A_1<$$
$$B_{i_1}<B_{i_1+1}<\cdots<B_{i_2-k_1-2}<A_2<$$
$$B_{i_2-k_1-1}<B_{i_2-k_1}<\cdots<B_{i_3-k_1-k_2-3}<A_3<$$
$$\cdots$$
$$B_{i_j-(k_1+k_2+\cdots+k_{j-1})-(j-1)}<B_{i_j-(k_1+k_2+\cdots+k_{j-1})-j+2}<$$
$$\cdots<B_{i_{j+1}-(k_1+k_2+\cdots+k_{j})-(j+1)}<A_{j+1}<$$
$$\cdots$$
$$B_{i_{t-1}-(k_1+k_2+\cdots+k_{t-2})-(t-2)}<B_{i_{t-1}-(k_1+k_2+\cdots+k_{t-2})-t+3}<$$
$$\cdots<B_{i_{t}-(k_1+k_2+\cdots+k_{t-1})-t}<A_{t}<$$
$$B_{i_{t}-(k_1+k_2+\cdots+k_{t-1})-t+1}<B_{i_{t}-(k_1+k_2+\cdots+k_{t-1})-t+2}<\cdots$$
$$<B_{i_{t}-(k_1+k_2+\cdots+k_{t-1})-t+1+n-(i_t+k_t)}(=B_p).$$
One can readily check that this map is bijective.
\end{proof}

\begin{example}
Suppose that $n=8$, $t=2$, and $(k_1,k_2)=(2,3)$. Therefore $p=2.$ Table \ref{bjjtb} illustrates the bijection between $P_{(2,3)}$ and $Q_{2,2}$.

\begin{table}
\begin{tabular}{|c|c|}
  \hline
  $P_{(2,3)}$ & $Q_{2,2}$ \\
  $(1,4)$ & $A_1<A_2<B_1<B_2$ \\
  $(1,5)$ & $A_1<B_1<A_2<B_2$ \\
  $(1,6)$ & $A_1<B_1<B_2<A_2$ \\
  $(2,5)$ & $B_1<A_1<A_2<B_2$ \\
  $(2,6)$ & $B_1<A_1<B_2<A_2$ \\
  $(3,6)$ & $B_1<B_2<A_1<A_2$ \\
  \hline
\end{tabular}
\caption{Example of bijection}\label{bjjtb}
\end{table}
\end{example}

\begin{theorem}\label{mainbijection}
There exists a $W-$equivariant bijection between the set of all biclosed sets in $\Phi$ (a finite irreducible root system of type $A_n$) and the set of quasitrivial semigroup structures on $X_{n+1}$.
\end{theorem}

\begin{proof}
Let $\Phi$ be a finite irreducible reduced root system of type $A_n$ and $\Phi^+, \Delta$ be one of its positive system and the corresponding simple system.
Let $\Delta_1, \Delta_2\subset \Delta$ be orthogonal.
Consider the set $$P_{\Delta_1,\Delta_2}=\{(\Delta_1',\Delta_2'): \Delta_1',\Delta_2'\subset \Delta, \Phi_{\Delta_1'}\simeq \Phi_{\Delta_1}, \Phi_{\Delta_2'}\simeq \Phi_{\Delta_2}\}.$$
Then one assumes that
$$\Delta_1\cup \Delta_2$$
$$=\{\alpha_{q}:q\in \bigcup_{j=1}^t\{i_j,i_{j}+1,\cdots,i_{j}+k_j-1\},$$
$$1\leq i_1, i_t+k_t-1\leq n, i_j+k_j+1\leq i_{j+1} \,\,\text{for}\,\, 1\leq j\leq t-1\}.$$
Then by Lemma \ref{bijection1}, there is a one-to-one correspondence between any two of $P_{\Delta_1,\Delta_2}$, $P_{(k_1,k_2,\cdots, k_t)}$  and $Q_{t,p}$.
If $\{\alpha_{q}:q\in \{i_j,i_{j}+1,\cdots,i_{j}+k_j-1\}\}\subset \Delta_1$ we write $\epsilon(j)=1$. Otherwise we write $\epsilon(j)=2.$

Given a biclosed set $\Psi^+_{\Xi_1,\Xi_2}$ in $\Phi$, we define a quasitrivial semigroup structure on $X_{n+1}=\{1,2,\cdots,n+1\}$.

Note that $\Psi^+_{\Xi_1,\Xi_2}=\sigma(\Phi^+_{\Delta_1,\Delta_2})$ for some $\sigma\in W=S_{n+1}$ where $\sigma$ is a permutation on $X_{n+1}$.
We define an equivalence relation $\thicksim$ on $X_{n+1}$. The equivalence classes are
$$A_1=\{\sigma(i_1),\sigma(i_1+1),\cdots,\sigma(i_1+k_1)\},$$
$$A_j=\{\sigma(i_j), \sigma(i_j+1), \cdots, \sigma(i_j+k_j)\}, 2\leq j\leq t,$$
Let $B'=\{1,2,\cdots,n+1\}\backslash \cup_{j=1}^t\sigma^{-1}(A_j)=\{l_1,l_2,\cdots,l_p\}$ with $l_1<l_2<\cdots<l_p$ and set
$$B_{1}=\{\sigma(l_1)\}, B_{2}=\{\sigma(l_2)\},\cdots$$
$$B_{p}=\{\sigma(l_p)\}.$$
Then by the above argument $\Phi_{\Delta_1,\Delta_2}^+$ induces a total order on $$ Q_{t,p}=\{A_1,A_2,\cdots,A_t,B_1,B_2,\cdots,B_p\}$$
 and a total preorder $\precsim$ on $X_{n+1}$.
For $x,y\in X_{n+1}$ and $x\nsim y$, we denote by $\max{\{x,y\}}$ the greater one of $x$ and $y$  under the total preorder $\precsim$.

Then a quasitrivial semigroup structure $F_{\Psi^+_{\Xi_1,\Xi_2}}: X_{n+1}\times X_{n+1}\rightarrow X_{n+1}$ is define as
\begin{equation}
  F(x,y) =
    \begin{cases}
      \max{\{x,y\}} & \text{if $x\nsim y$}\\
      x & \text{if $x\thicksim y$}, x,y\in A_j \,\text{and}\, \epsilon(j)=1\\
      y & \text{if $x\thicksim y$}, x,y\in A_j \,\text{and}\, \epsilon(j)=2
    \end{cases}
\end{equation}

Conversely given a quasitrivial semigroup structure $F: X_{n+1}\times X_{n+1}\rightarrow X_{n+1}$, we construct a biclosed set in $\Phi$.
First $F$ determines a total preorder $\precsim$ on $X_{n+1}$. Then $\precsim$ induces an equivalence relation on $X_{n+1}.$
Let $A_1, A_2, \cdots, A_t$ be the equivalence classes containing more than one element. Let $B_1,B_2,\cdots, B_p$ be the equivalence classes with one element. Further define $k_j=|A_j|-1, 1\leq j\leq t.$
On an equivalence class $A_j$ induced by this total preorder, either $F(x,y)=x$ everywhere or $F(x,y)=y$ everywhere.
Suppose that $\epsilon(j)=1$ if $F(x,y)=x$ for $x,y\in A_j$ and $\epsilon(j)=2$ otherwise. The total preorder $\precsim$ yields a total order $\prec$ on
$\{A_1, A_2, \cdots, A_t, B_1, B_2, \cdots,B_p\}$. By Lemma \ref{bijection1}, $\prec$ gives a vector of positive integer
$$(i_1,i_2,\cdots,i_t):1\leq i_1, i_t+k_t-1\leq n, i_j+k_j+1\leq i_{j+1} \,\,\text{for}\,\, 1\leq j\leq t-1.$$
Define
$$\Delta_1=\{\alpha_{q}:q\in \bigcup_{1\leq j\leq t, \epsilon(j)=1}\{i_j,i_{j}+1,\cdots,i_{j}+k_j-1\}\}.$$
$$\Delta_2=\{\alpha_{q}:q\in \bigcup_{1\leq j\leq t, \epsilon(j)=2}\{i_j,i_{j}+1,\cdots,i_{j}+k_j-1\}\}.$$
Extend the total preorder $\precsim$ to a total order $\prec'$ on  $X_{n+1}$.
Consider the permutation $\sigma$ on $X_{n+1}$ which carries $u,1\leq u\leq n+1$ to the $u$-th element of the chain $(X_{n+1},\prec')$.
Then  $\sigma(\Phi^+_{\Delta_1,\Delta_2})$ is the desired biclosed set. Note that although the total order $\prec'$ is not unique (since restricted to an equivalence class the total order $\prec'$ has multiple choices), the biclosed set we obtain is independent of the choice, since the various permutations induced by $\prec'$ fall in the stabilizer of the resulting biclosed set.

The two processes are clearly inverse to each other by construction.

Now we show that the bijection is $W-$invariant.
Let $B=\Psi^+_{\Delta_1',\Delta_2'}$. Then there exists $v\in W$ such that $B=v(\Phi^+_{\Delta_1,\Delta_2}).$
Since $w(v(\Phi^+_{\Delta_1,\Delta_2}))=wv(\Phi^+_{\Delta_1,\Delta_2})$ , it suffices to show that for biclosed set $B=\Phi^+_{\Delta_1,\Delta_2}$ one has
$F_{wB}=wF_B.$ Using the notation preceding the theorem,  the quasitrivial semigroup structure $F_B$ first determines an equivalence relation $\sim$ on $X_{n+1}$ with the equivalence classes
$$A_1,A_2,\cdots,A_t,B_1,B_2\cdots,B_p$$
where $A_i$'s are equivalence classes with cardinality greater than 1 and $B_i$'s are singleton equivalence classes.
Second, $F_B$ determines a total order on the set $\{A_1,A_2,\cdots, A_t, B_1,B_2,\cdots,B_p\}$. Without loss of generality we can assume that
$$B_1<B_2<\cdots <B_{i_1-1}<A_1<$$
$$B_{i_1}<B_{i_1+1}<\cdots<B_{i_2-k_1-2}<A_2<$$
$$\cdots$$
$$B_{i_{t-1}-(k_1+k_2+\cdots+k_{t-2})-(t-2)}<B_{i_{t-1}-(k_1+k_2+\cdots+k_{t-2})-t+3}<$$
$$\cdots<B_{i_{t}-(k_1+k_2+\cdots+k_{t-1})-t}<A_{t}<$$
$$B_{i_{t}-(k_1+k_2+\cdots+k_{t-1})-t+1}<B_{i_{t}-(k_1+k_2+\cdots+k_{t-1})-t+2}<\cdots<B_p.$$
The biclosed set  $B$ also determines an function $\epsilon$ such that $F|_{A_j}=\pi_1$ if $\epsilon(j)=1$ and $F|_{A_j}=\pi_2$ if $\epsilon(j)=2.$
By definition $(wF_B)(x,y)=w(F_B(w^{-1}(x),w^{-1}(y))$. Therefore the equivalence classes and the total order on them determined by $wF_B$ are
$\{w(A_1),w(A_2),\cdots, w(A_t), w(B_1),w(B_2),\cdots,w(B_p)\}$
and
$$w(B_1)<w(B_2)<\cdots <w(B_{i_1-1})<w(A_1)<$$
$$w(B_{i_1})<w(B_{i_1+1})<\cdots<w(B_{i_2-k_1-2})<w(A_2)<$$
$$\cdots$$
$$w(B_{i_{t-1}-(k_1+k_2+\cdots+k_{t-2})-(t-2)})<w(B_{i_{t-1}-(k_1+k_2+\cdots+k_{t-2})-t+3})<$$
$$\cdots<w(B_{i_{t}-(k_1+k_2+\cdots+k_{t-1})-t})<w(A_{t})<$$
$$w(B_{i_{t}-(k_1+k_2+\cdots+k_{t-1})-t+1})<w(B_{i_{t}-(k_1+k_2+\cdots+k_{t-1})-t+2})<\cdots<w(B_p).$$
We also have that $F|_{w(A_j)}=\pi_1$ if $\epsilon(j)=1$ and $F|_{w(A_j)}=\pi_2$ if $\epsilon(j)=2.$ By the construction described in this proof, this corresponds exactly to $F_{wB}$.
\end{proof}

\begin{example}
In this example  the bijection between the biclosed sets in the root system of $A_2$ and quasitrivial semigroup structures on $\{1,2,3\}$ is presented in detail. (See Table \ref{typeatable})
For an equivalence class $A\subset \{1,2,3\}$, we write $A^i, i=1,2$ if the quasitrivial semigroup structure restricted to $A$ is $\pi_i$.

\begin{table}\tiny
\begin{tabular}{|c|c|}
  \hline
  Biclosed set & Quasitrivial semigroup structure \\
    \hline
  $\Phi^+_{\{\alpha_1,\alpha_2\},\emptyset}=\emptyset$ & $\{1,2,3\}^1$ ($\{1,2,3\}=A_1$) \\
  $\Phi^+_{\{\alpha_1\},\emptyset}=\{\alpha_1+\alpha_2,\alpha_2\}$ & $\{1,2\}^1\prec 3$ ($\{1,2\}=A_1$, $\{3\}=B_1$)\\
  $\Phi^+_{\{\alpha_2\},\emptyset}=\{\alpha_1+\alpha_2,\alpha_1\}$ & $1\prec \{2,3\}^1$ ($\{2,3\}=A_1$, $\{1\}=B_1$)\\
  $(2,3)\Phi^+_{\{\alpha_1\},\emptyset}=\{-\alpha_2,\alpha_1\}$ & $\{1,3\}^1\prec 2$ ($\{(2,3)1,(2,3)2\}=A_1$, $\{(2,3)3\}=B_1$)\\
  $(1,3)\Phi^+_{\{\alpha_1\},\emptyset}=\{-\alpha_1-\alpha_2,-\alpha_1\}$ & $\{2,3\}^1\prec 1$ ($\{(1,3)1,(1,3)2\}=A_1$, $\{(1,3)3\}=B_1$)\\
  $(1,2)\Phi^+_{\{\alpha_2\},\emptyset}=\{\alpha_2,-\alpha_1\}$ & $2\prec \{1,3\}^1$ ($\{(1,2)2,(1,2)3\}=A_1$, $\{(1,2)1\}=B_1$)\\
  $(1,3)\Phi^+_{\{\alpha_2\},\emptyset}=\{-\alpha_1-\alpha_2,-\alpha_2\}$ & $3\prec \{1,2\}^1$ ($\{(1,3)2,(1,3)3\}=A_1$, $\{(1,3)1\}=B_1$)\\
  $\Phi^+_{\emptyset,\emptyset}=\{\alpha_1,\alpha_1+\alpha_2,\alpha_2\}$ & $1\prec 2\prec 3, B_i=\{i\},1\leq i\leq 3$\\
  $(1,2)\Phi^+_{\emptyset,\emptyset}=\{-\alpha_1,\alpha_1+\alpha_2,\alpha_2\}$ & $2\prec 1\prec 3, B_i=\{(1,2)i\},1\leq i\leq 3$\\
  $(2,3)\Phi^+_{\emptyset,\emptyset}=\{-\alpha_2,\alpha_1+\alpha_2,\alpha_1\}$ & $1\prec 3\prec 2, B_i=\{(2,3)i\},1\leq i\leq 3$\\
  $(1,3)\Phi^+_{\emptyset,\emptyset}=\{-\alpha_2,-\alpha_1-\alpha_2,-\alpha_1\}$ & $3\prec 2\prec 1, B_i=\{(1,3)i\},1\leq i\leq 3$\\
$(1,2,3)\Phi^+_{\emptyset,\emptyset}=\{\alpha_2,-\alpha_1-\alpha_2,-\alpha_1\}$ & $2\prec 3\prec 1, B_i=\{(1,2,3)i\},1\leq i\leq 3$\\
$(1,3,2)\Phi^+_{\emptyset,\emptyset}=\{-\alpha_2,-\alpha_1-\alpha_2,\alpha_1\}$ & $3\prec 1\prec 2, B_i=\{(1,3,2)i\},1\leq i\leq 3$\\
$\Phi^+_{\emptyset,\{\alpha_1\}}=\{\alpha_1+\alpha_2,\alpha_2,\pm\alpha_1\}$ & $\{1,2\}^2\prec 3$ ($\{1,2\}=A_1$, $\{3\}=B_1$)\\
$\Phi^+_{\emptyset,\{\alpha_2\}}=\{\alpha_1+\alpha_2,\alpha_1,\pm\alpha_2\}$ & $1\prec\{2,3\}^2$ ($\{2,3\}=A_1$, $\{1\}=B_1$)\\
$(2,3)\Phi^+_{\emptyset,\{\alpha_1\}}=\{-\alpha_2,\alpha_1,\pm(\alpha_1+\alpha_2)\}$ & $\{1,3\}^2\prec 2$ ($\{(2,3)1,(2,3)2\}=A_1$, $\{(2,3)3\}=B_1$)\\
 $(1,3)\Phi^+_{\emptyset,\{\alpha_1\}}=\{-\alpha_1-\alpha_2,-\alpha_1,\pm\alpha_2\}$ & $\{2,3\}^2\prec 1$ ($\{(1,3)1,(1,3)2\}=A_1$, $\{(1,3)3\}=B_1$)\\
$(1,2)\Phi^+_{\emptyset,\{\alpha_2\}}=\{\alpha_2,-\alpha_1,\pm(\alpha_1+\alpha_2)\}$ & $2\prec \{1,3\}^2$ ($\{(1,2)2,(1,2)3\}=A_1$, $\{(1,2)1\}=B_1$)\\
 $(1,3)\Phi^+_{\emptyset,\{\alpha_2\}}=\{-\alpha_1-\alpha_2,-\alpha_2.\pm\alpha_1\}$ & $3\prec \{1,2\}^2$ ($\{(1,3)2,(1,3)3\}=A_1$, $\{(1,3)1\}=B_1$)\\
 $\Phi^+_{\emptyset,\{\alpha_1,\alpha_2\}}=\Phi$ & $\{1,2,3\}^2$ ($\{1,2,3\}=A_1$)\\
\hline
\end{tabular}
\caption{Biclosed sets in a type $A_2$ root system and quasitrivial semigroups structures}\label{typeatable}
\end{table}
\end{example}

\begin{proposition}\label{variouscorr}
\begin{enumerate}
\item The commutative quasitrivial semigroup structures on $X_{n+1}$ are in bijection with the positive systems of $\Phi$.
\item The anticommutative quasitrivial semigroup structures on $X_{n+1}$ are precisely $F_{\emptyset}=F_{\Phi^+_{\Delta,\emptyset}}$ and $F_{\Phi}=F_{\Phi^+_{\emptyset,\Delta}}$.
\item The quasitrivial semigroup structures on $X_{n+1}$ with an identity element (resp. zero element) correspond to the biclosed sets of the form $w\Phi^+_{\Delta_1,\Delta_2}$ where $\alpha_1\not\in \Delta_1\cup \Delta_2$ (resp. $\alpha_n\not\in \Delta_1\cup \Delta_2$).
\item The parabolic (resp. horocyclic) subsets of $\Phi$ correspond to quasitrivial semigroup structures $F$ such that $F|_{\{x,y\}}\neq \pi_1$ (resp. $F|_{\{x,y\}}\neq \pi_2$) for any $x,y\in X_{n+1}$.
\end{enumerate}
\end{proposition}

\begin{proof}
(1) For a commutative quasitrivial semigroup structure $F: X_{n+1}\times X_{n+1}\rightarrow X_{n+1}$, the associated total preorder $\precsim$ is a total order on $X_{n+1}$.
By Theorem \ref{mainbijection}, these correspond to biclosed sets of the form $w\Phi^+_{\emptyset,\emptyset}=w\Phi^+$ for some $w\in W$. These are exactly the positive system of $\Phi$.

(2) For an anticommutative quasitrivial semigroup structure $F: X_{n+1}\times X_{n+1}\rightarrow X_{n+1}$, the associated total preorder $\precsim$ has exactly one equivalence class on $X_{n+1}$.
By Theorem \ref{mainbijection}, these correspond to $\Phi^+_{\Delta,\emptyset}$ and $\Phi^+_{\emptyset,\Delta}$.

(3) The quasitrivial semigroup has an identity element (resp. zero element) if and only if the minimal (resp. maximal) class of the induced total preorder is a singleton set. In view of the construction preceding Theorem \ref{mainbijection} and Lemma \ref{bijection1}, this implies that $i_1>1$ (resp. $i_t<n$). Therefore the Proposition (3) follows.

(4) This is clear from the definition of a parabolic set and the bijection.
\end{proof}

Now we introduce a lemma which gives another easy way to reconstruct the biclosed set based on a given quasitrivial semigroup structure.

For $j\leq k$, let $\alpha_{j,k+1}:=\alpha_j+\alpha_{j+1}+\cdots+\alpha_k$ be a root. (Note that the second index of $\alpha_{j,k+1}$ is set to be $k+1$ on purpose. This root corresponds to the transposition $(j,k+1)\in S_{n+1}$ under the canonical bijection between the positive roots and transpositions.)

\begin{lemma}\label{twoelt}
Let $C$ be a biclosed set in $\Phi$ and $F_C$ be the corresponding quasitrivial associative binary operation.
\begin{enumerate}
\item $\{\alpha_{j,k+1},-\alpha_{j,k+1}\}\cap C=\emptyset$ if and only if $F_{C}(j,k+1)=j, F_{C}(k+1,j)=k+1.$
\item $\alpha_{j,k+1}\in C$ and $-\alpha_{j,k+1}\not\in C$ if and only if  $F_{C}(j,k+1)=F_{C}(k+1,j)=k+1.$
\item $-\alpha_{j,k+1}\in C$ and $\alpha_{j,k+1}\not\in C$ if and only if $F_{C}(j,k+1)=F_{C}(k+1,j)=j.$
\item $\{\alpha_{j,k+1},-\alpha_{j,k+1}\}\subset C$ if and only if $F_{C}(j,k+1)=k+1, F_{C}(k+1,j)=j.$
\end{enumerate}
\end{lemma}

\begin{proof}
We prove (1) and the other items can be verified similarly.
First we note that if we write $\alpha_{k+1,j}=-\alpha_{j,k+1}$, then for $w\in W$, $w\alpha_{j,k+1}=\alpha_{w(j),w(k+1)}$ where on the right-hand side $w$ is regarded as a permutation.

Assume that $F_{C}(j,k+1)=j, F_{C}(k+1,j)=k+1$. Then $j$ and $k+1$ are contained in an equivalence class $A_m$ induced by $F_C$ and $\epsilon(m)=1.$
Suppose that $C=w\Phi^+_{\Delta_1,\Delta_2}$.
Then by the bijection described in the proof of Theorem \ref{mainbijection}, $A_m=\{w(i_m),\cdots,w(i_m+k_m) \}$.
Then the above condition implies that
$$w\Phi_{\{\alpha_{i_m},\alpha_{i_{m}+1},\cdots ,\alpha_{i_m+k_m-1}\}}\cap C=\emptyset,$$
$$w(i_m+x)=j, w(i_m+y)=k+1 \,\,\text{for some}\,\, 0\leq x,y\leq k_m.$$
Again the condition implies that neither $w(\alpha_{i_m+x,i_m+y})=\alpha_{j,k+1}$ nor $w(\alpha_{i_m+y,i_m+x})=\alpha_{k+1,j}$ is contained in $C$.

Conversely, suppose that $\{\alpha_{j,k+1},-\alpha_{j,k+1}\}\cap C=\emptyset$. Also assume that $C=w\Phi^+_{\Delta_1,\Delta_2}$.
Then neither $\alpha_{w^{-1}(j),w^{-1}(k+1)}$ nor $\alpha_{w^{-1}(k+1),w^{-1}(j)}$ is contained in $\Phi^+_{\Delta_1,\Delta_2}(=w^{-1}C)$.
This implies that there exists $$\{\alpha_{i_m},\alpha_{i_{m}+1},\cdots ,\alpha_{i_m+k_m-1}\}\subset \Delta_1$$ and $$w^{-1}(j),w^{-1}(k+1)\in \{i_m,\cdots,i_m+k_m\}.$$
Then there exists an equivalence class
$$\{i_m,\cdots,i_m+k_m\}$$
induced by $F_{w^{-1}C}$ containing both $w^{-1}(j)$ and $w^{-1}(k+1).$ Such an equivalence class is mapped to $1$ by $\epsilon.$
So $F_{w^{-1}C}(w^{-1}(j),w^{-1}(k+1))=w^{-1}(j)$ and $F_{w^{-1}C}(w^{-1}(k+1),w^{-1}(j))=w^{-1}(k+1).$
Then our assertion follows from the fact $F$ is $W-$equivariant and the formula $(wF_B)(x,y)=w(F_B(w^{-1}(x),w^{-1}(y))$.
\end{proof}

The following theorem is the consequence of the main theorem and the above lemma.

Let $F$ be a binary relation: $X_{n+1}\times X_{n+1}\rightarrow X_{n+1}$ such that $F(a,b)\in \{a,b\}$ for all $a,b\in X_{n+1}.$ Then define a set $C\subset \Phi$ in the following way:
\begin{enumerate}
\item $\{\alpha_{j,k+1},-\alpha_{j,k+1}\}\cap C=\emptyset$ if and only if $F(j,k+1)=j, F(k+1,j)=k+1.$
\item $\alpha_{j,k+1}\in C$ and $-\alpha_{j,k+1}\not\in C$ if and only if  $F(j,k+1)=F(k+1,j)=k+1.$
\item $-\alpha_{j,k+1}\in C$ and $\alpha_{j,k+1}\not\in C$ if and only if $F(j,k+1)=F(k+1,j)=j.$
\item $\{\alpha_{j,k+1},-\alpha_{j,k+1}\}\subset C$ if and only if $F(j,k+1)=k+1, F(k+1,j)=j.$
\end{enumerate}

\begin{theorem}
The set $C$ is biclosed (in $\Phi$) if and only if $F$ is associative.
\end{theorem}

Finally, we introduce a partial order  on the set of the quasitrivial semigroups structures on $X_{n+1}$.
Denote by $F_{n+1}$ the set $$\{F: X_{n+1}\times X_{n+1}\rightarrow X_{n+1}:\, F\, \text{is associative and quasitrivial}\}.$$
First, we define a partial order on $F_2$. There are 4 elements in $F_2.$
$$F_{(1)}=\pi_1, F_{(2)}(a,b)=\max{\{a,b\}}, F_{(3)}(a,b)=\min{\{a,b\}}, F_{(4)}=\pi_2.$$
(Here the functions $\max$ and $\min$ are defined in terms of the usual order of the natural numbers.)
Define a partial order on $F_2$: $F_{(1)}<F_{(2)}, F_{(1)}<F_{(3)}, F_{(1)}<F_{(4)}, F_{(2)}<F_{(4)}, F_{(3)}<F_{(4)}.$

Now for $F,F'\in F_{n+1}$, we define $F<F'$ if and only if $F|_{\{a,b\}}<F'|_{\{a,b\}}$ for any $a\neq b, a,b\in X_{n+1}$.

\begin{lemma}
Let $C_1,C_2$ be two biclosed sets in $\Phi$. Then $C_1\subset C_2$ if and only if $F_{C_1}\leq F_{C_2}$. Therefore $(F_{n+1},\leq)$ is a well-defined poset isomorphic to $(\mathcal{B}(\Phi),\subset)$.
\end{lemma}

\begin{proof}
This follows directly from Lemma \ref{twoelt} and the definition of the partial order $\leq$.
\end{proof}

\begin{remark}
In \cite{DyerReflOrder} it is shown that $(\mathcal{B}(\Phi),\subset)$ (and therefore our poset $(F_{n+1},\leq)$) is actually a lattice.
\end{remark}

\section{bijection for biclosed sets in standard parabolic subsets}\label{parabolicsec}

Let $\Phi^+$ be the chosen standard positive system of a finite irreducible root system $\Phi$. Let $\Delta$ be its simple system and let $\Delta'\subset \Delta.$ Then we call the set $\Phi^+\cup \Phi_{\Delta'}$ a standard parabolic subset of $\Phi$. Indeed a standard parabolic subset is a biclosed set in $\Phi$ and can be denoted by $\Phi^+_{\emptyset, \Delta'}.$ It is conjectured in \cite{DyerWeakOrder} that the biclosed sets in a standard parabolic subset form a lattice under the inclusion. Such a partial order is called a parabolic weak order there. Therefore it is desirable to provide  a combinatorial description of the biclosed sets in a standard parabolic subset. In this section we extend the bijection discussed in Section \ref{mainbijA} to establish a bijection of quasitrivial semigroup structures and biclosed sets in a standard parabolic subset.

\begin{lemma}\label{parabolic}
The biclosed sets in $\Phi^+_{\emptyset,\Delta_2}$ are in bijection with the biclosed sets in $\Phi$ whose stabilizer is contained in $W_{\Delta_2}$, the parabolic subgroup of $W$ generated by $s_{\alpha},\alpha\in \Delta_2.$ More specifically given a biclosed set $C$ in $\Phi^+_{\emptyset,\Delta_2}$, one associates with it  a biclosed set (in $\Phi$)
$$(C\cap \Phi_{\Delta_2})\cup(C\cap \Phi^+_{\Delta_2,\emptyset})\cup -(\Phi^+_{\Delta_2,\emptyset}\backslash (\Phi^+_{\Delta_2,\emptyset}\cap C)).$$
Conversely for a biclosed set $F$ in $\Phi$, one associates with it a biclosed set $F\cap \Phi^+_{\emptyset,\Delta_2}$ in  $\Phi^+_{\emptyset,\Delta_2}$.
\end{lemma}

\begin{proof}
Given a biclosed set $C$ in $\Phi^+_{\emptyset,\Delta_2}$. Write $B:=C\cap \Phi_{\Delta_2}$, $D:=C\cap \Phi^+_{\Delta_2,\emptyset}$ and $$E:=(C\cap \Phi_{\Delta_2})\cup(C\cap \Phi^+_{\Delta_2,\emptyset})\cup -(\Phi^+_{\Delta_2,\emptyset}\backslash (\Phi^+_{\Delta_2,\emptyset}\cap C))$$$$=B\cup D\cup -(\Phi^+_{\Delta_2,\emptyset}\backslash D).$$

First we show that $E$ is closed. The subset $B$ is the intersection of a closed set $C$ with a (closed) root subsystem and therefore closed.
The subset  $D$ is also the intersection of two closed sets and therefore closed. The subset $\Phi^+_{\Delta_2,\emptyset}\backslash D=(\Phi^+_{\emptyset,\Delta_2}\backslash C)\cap \Phi^+_{\Delta_2,\emptyset}$ is again the intersection of two closed sets and is thus closed.

Take $\alpha\in B,
\beta\in D$ with $\alpha+\beta\in \Phi$. Since $\alpha,\beta\in C$, $\alpha+\beta\in C$ since $C$ is closed. Also since $\alpha,\beta\in \Phi^+_{\emptyset,\Delta_2}$, their sum is also contained in $\Phi^+_{\emptyset,\Delta_2}$. So $\alpha+\beta\in B\cup D.$

Take $\alpha\in B, \beta\in -(\Phi^+_{\Delta_2,\emptyset}\backslash D)$ with $\alpha+\beta\in \Phi.$ Since $\alpha,\beta\in \Phi^-_{\emptyset,-\Delta_2}$ one has $\alpha+\beta\in \Phi^-_{\emptyset,-\Delta_2}.$ If $\alpha+\beta=\gamma\in \Phi_{\Delta_2}$, since $-\alpha\in \Phi_{\Delta_2}$ one has $\beta\in \Phi_{\Delta_2}$. A contradiction. Therefore $\alpha+\beta\in \Phi^-_{\emptyset,-\Delta_2}\backslash \Phi_{\Delta_2}=\Phi^-_{-\Delta_2,\emptyset}=-\Phi^+_{\Delta_2,\emptyset}.$
Suppose that $\alpha+\beta\in -D.$ Then $-\alpha-\beta\in D\subset C$. But $\alpha\in C$. The closedness of $C$ implies that $-\beta\in C$. A contradiction. So $\alpha+\beta\in -(\Phi^+_{\Delta_2,\emptyset}\backslash D).$

Finally we take $\alpha\in D,\beta\in -(\Phi^+_{\Delta_2,\emptyset}\backslash D)$ with $\alpha+\beta\in \Phi.$
Note that $(C\cap \Phi_{\Delta_2}^+)\cup D(=C\cap \Phi^+)$ is a biclosed set in $\Phi^+$. Hence
$$\Psi_1^+:=(C\cap \Phi_{\Delta_2}^+)\cup D\cup -(\Phi^+_{\Delta_2,\emptyset}\backslash D)\cup (\Phi_{\Delta_2}^-\backslash -C)$$
is another positive system.
So $\alpha,\beta\in \Psi_1^+$.
 On the other hand $\Psi^+:=\Phi_{\Delta_2}^-\cup \Phi^+_{\Delta_2,\emptyset}=v\Phi^+$ where $v$ is the longest element of the standard parabolic subgroup generated by the reflections in $\Phi_{\Delta_2}$ is a positive system contained in $\Phi^+_{\emptyset,\Delta_2}$.
 Therefore $C\cap \Psi^+=(C\cap \Phi_{\Delta_2}^-)\cup D$ is a biclosed set in $\Psi^+.$

 Consequently $\alpha,\beta\in \Psi_2^+:=(C\cap \Phi_{\Delta_2}^-)\cup D\cup\backslash -(\Phi^+_{\Delta_2,\emptyset}\backslash D)\cup (\Phi_{\Delta_2}^+\backslash -C)$ which is  another positive system. We illustrate the various sets in Figure \ref{coloredsets}.
 
 \begin{figure}\label{coloredsets}
\center
\tikzset{every picture/.style={line width=0.75pt}} 

\begin{tikzpicture}[x=0.75pt,y=0.75pt,yscale=-1,xscale=1]

\draw  [fill={rgb, 255:red, 189; green, 16; blue, 224 }  ,fill opacity=1 ] (207.33,63.47) -- (294.33,63.47) -- (294.33,129.47) -- (207.33,129.47) -- cycle ;
\draw  [fill={rgb, 255:red, 245; green, 166; blue, 35 }  ,fill opacity=1 ] (207.33,102.47) -- (245.33,102.47) -- (245.33,129.47) -- (207.33,129.47) -- cycle ;
\draw  [fill={rgb, 255:red, 65; green, 117; blue, 5 }  ,fill opacity=1 ] (294.33,63.47) -- (381.33,63.47) -- (381.33,129.47) -- (294.33,129.47) -- cycle ;
\draw  [fill={rgb, 255:red, 40; green, 35; blue, 245 }  ,fill opacity=1 ] (343.33,102.47) -- (381.33,102.47) -- (381.33,129.47) -- (343.33,129.47) -- cycle ;
\draw  [fill={rgb, 255:red, 219; green, 224; blue, 16 }  ,fill opacity=1 ] (294.33,129.47) -- (381.33,129.47) -- (381.33,195.47) -- (294.33,195.47) -- cycle ;
\draw  [fill={rgb, 255:red, 35; green, 228; blue, 245 }  ,fill opacity=1 ] (343.33,129.47) -- (381.33,129.47) -- (381.33,156.47) -- (343.33,156.47) -- cycle ;
\draw  [fill={rgb, 255:red, 255; green, 255; blue, 255 }  ,fill opacity=1 ] (207.33,129.47) -- (294.33,129.47) -- (294.33,195.47) -- (207.33,195.47) -- cycle ;
\draw  [fill={rgb, 255:red, 155; green, 155; blue, 155 }  ,fill opacity=1 ] (207.33,129.47) -- (245.33,129.47) -- (245.33,156.47) -- (207.33,156.47) -- cycle ;
\draw  [fill={rgb, 255:red, 189; green, 16; blue, 224 }  ,fill opacity=1 ] (416.33,60.47) -- (503.33,60.47) -- (503.33,126.47) -- (416.33,126.47) -- cycle ;
\draw  [fill={rgb, 255:red, 245; green, 166; blue, 35 }  ,fill opacity=1 ] (416.33,99.47) -- (454.33,99.47) -- (454.33,126.47) -- (416.33,126.47) -- cycle ;
\draw  [fill={rgb, 255:red, 224; green, 40; blue, 16 }  ,fill opacity=1 ] (503.33,60.47) -- (590.33,60.47) -- (590.33,126.47) -- (503.33,126.47) -- cycle ;
\draw  [fill={rgb, 255:red, 35; green, 245; blue, 72 }  ,fill opacity=1 ] (552.33,99.47) -- (590.33,99.47) -- (590.33,126.47) -- (552.33,126.47) -- cycle ;
\draw  [fill={rgb, 255:red, 139; green, 87; blue, 42 }  ,fill opacity=1 ] (503.33,126.47) -- (590.33,126.47) -- (590.33,192.47) -- (503.33,192.47) -- cycle ;
\draw  [fill={rgb, 255:red, 0; green, 0; blue, 0 }  ,fill opacity=1 ] (552.33,126.47) -- (590.33,126.47) -- (590.33,153.47) -- (552.33,153.47) -- cycle ;
\draw  [fill={rgb, 255:red, 255; green, 255; blue, 255 }  ,fill opacity=1 ] (416.33,126.47) -- (503.33,126.47) -- (503.33,192.47) -- (416.33,192.47) -- cycle ;
\draw  [fill={rgb, 255:red, 155; green, 155; blue, 155 }  ,fill opacity=1 ] (416.33,126.47) -- (454.33,126.47) -- (454.33,153.47) -- (416.33,153.47) -- cycle ;

\draw (221,106) node [anchor=north west][inner sep=0.75pt]   [align=left] {$\displaystyle D$};
\draw (345.33,105.47) node [anchor=north west][inner sep=0.75pt]   [align=left] {$\displaystyle -B^{-}$};
\draw (357,134) node [anchor=north west][inner sep=0.75pt]   [align=left] {$\displaystyle B^{-}$};
\draw (215,136) node [anchor=north west][inner sep=0.75pt]   [align=left] {$\displaystyle -D$};
\draw (430,103) node [anchor=north west][inner sep=0.75pt]   [align=left] {$\displaystyle D$};
\draw (566,106) node [anchor=north west][inner sep=0.75pt]   [align=left] {$\displaystyle B^{+}$};
\draw (554.33,129.47) node [anchor=north west][inner sep=0.75pt]   [align=left] {$\displaystyle \textcolor[rgb]{1,1,1}{-B^{+}}$};
\draw (424,133) node [anchor=north west][inner sep=0.75pt]   [align=left] {$\displaystyle -D$};

\end{tikzpicture}
\caption{\label{coloredsets} Colored area representation of the sets involved.}
\end{figure}

In Figure \ref{coloredsets}, the orange area represents $D$, the purple area represents $\Phi^+_{\Delta_2,\emptyset}\backslash D,$ the grey area represents $-D$, the white area represents $-(\Phi^+_{\Delta_2,\emptyset}\backslash D)$, the light blue area represents $B^-:=C\cap \Phi^-_{\Delta_2}$, the dark blue area represents $-B^-=\Phi^+_{\Delta_2}\cap -C$, the yellow area represents $\Phi_{\Delta_2}^-\backslash C,$ the dark green area represents $\Phi^+_{\Delta_2}\backslash -C,$ the light green area represents $B^+:=\Phi_{\Delta_2}^+\cap C$, the black area represents $-B^+=\Phi^-_{\Delta_2}\cap -C,$
the red area represents $\Phi^+_{\Delta_2}\backslash C$ and the brown area represents $\Phi^-_{\Delta_2}\backslash -C.$

Therefore the set $E$ corresponds to the union of the orange area, the light green area, the light blue area and white area.
The positive system $\Psi^+_1$ corresponds to the union of the orange area, white area, light green area and brown area.
The positive system $\Psi^+_2$ corresponds to the union of the orange area, the white area, the light blue area and the dark green area.

 So $\alpha_1+\alpha_2\in \Psi_1^+\cap \Psi_2^+.$ From this one sees that $\alpha+\beta$ has to be contained in one of the following three sets

(1)  $-(\Phi^+_{\Delta_2,\emptyset}\backslash D)\cup D$ (union of orange and white areas);

(2) $(C\cap \Phi_{\Delta_2}^+)\cap (\Phi_{\Delta_2}^+\backslash -C),$ (contained in the light green area);

(3) $(C\cap \Phi_{\Delta_2}^-)\cap (\Phi_{\Delta_2}^-\backslash -C),$ (contained in the light blue area).

Therefore $\alpha+\beta\in B\cup D\cup -(\Phi^+_{\Delta_2,\emptyset}\backslash D).$

Now we show that $\Phi\backslash E$ is also closed. This follows from the fact that 
$$\Phi\backslash E=((\Phi^+_{\emptyset,\Delta_2}\backslash C)\cap \Phi_{\Delta_2})\cup((\Phi^+_{\emptyset,\Delta_2}\backslash C)\cap \Phi^+_{\Delta_2,\emptyset})$$
$$\cup -(C\cap \Phi^+_{\Delta_2,\emptyset}),$$
the fact that  $\Phi^+_{\emptyset,\Delta_2}\backslash C$ is biclosed in $\Phi^+_{\emptyset,\Delta_2}$ and the above argument showing that $E$ is closed.

The stabilizer of $E$ lies in $\{s_{\alpha}|\,\alpha\in E\cap -E\}\subset W_{\Delta_2}.$

Finally, the associations $C\mapsto E$ and $E\mapsto E\cap \Phi^+_{\emptyset,\Delta_2}$ are clearly inverse to each other.
\end{proof}

\begin{remark}
\begin{enumerate}
\item The above lemma and its proof both work for other finite crystallographic root systems other than those of type $A_n$.
\item This lemma can be viewed as a generalization of the well-known fact that the biclosed sets in a given positive system $\Psi^+$ are in bijection with the positive systems (i.e. biclosed set in $\Phi$ with trivial stabilizer) via the map $B\mapsto B\cup -(\Psi^+\backslash B)$.
\end{enumerate}
\end{remark}

Now we assume that $\Delta_2=\bigcup_{i=1}^m\Delta_{2,i}$ where $\Delta_{2,i}=\{\alpha_{i_1},\alpha_{i_1+1},\cdots,\alpha_{i_1+t_i}\}$ and $|i_1+t_i-i'_1|\geq 2$ for $i\neq i'$ (i.e. the root subsystems $\Phi_{\Delta_{2,i}}$ are pairwise disjoint and their union is $\Phi_{\Delta_2}$). The following theorem is the direct consequence of Lemma \ref{parabolic}, Lemma \ref{twoelt} and Theorem \ref{mainbijection}. Again for $j\leq k$, let $\alpha_{j,k+1}:=\alpha_j+\alpha_{j+1}+\cdots+\alpha_k$ be a root.

\begin{theorem}\label{bjpara}
\begin{enumerate}
\item The biclosed sets in $\Phi^+_{\emptyset,\Delta_2}$ are in bijection with quasitrivial semigroup structures on $X_{n+1}$ with the property that
each equivalence class induced by the quasitrivial semigroup structure is contained in the set $\{i_1,i_1+1,\cdots, i_1+t_i+1\}$ for some $i, 1\leq i\leq m$.
\item Let $(X_{n+1},F)$ be a quasitrivial semigroup with the property described in (1). Then the corresponding biclosed set $C$ in $\Phi^+_{\emptyset,\Delta_2}$ can be constructed  as follows:
    \begin{enumerate}
\item $\{\alpha_{j,k+1},-\alpha_{j,k+1}\}\cap C=\emptyset$ if  $F(j,k+1)=j, F(k+1,j)=k+1$.
\item $\alpha_{j,k+1}\in C$ and $-\alpha_{j,k+1}\not\in C$ if   $F(j,k+1)=F(k+1,j)=k+1$.
\item $-\alpha_{j,k+1}\in C$ and $\alpha_{j,k+1}\not\in C$ if $F(j,k+1)=F(k+1,j)=j$ and $j,k+1\in \{i_1,i_1+1,\cdots, i_1+t_i+1\}$ for some $i, 1\leq i\leq m$.
\item $\{\alpha_{j,k+1},-\alpha_{j,k+1}\}\subset C$ if  $F(j,k+1)=k+1, F(k+1,j)=j.$
\item $\{\alpha_{j,k+1},-\alpha_{j,k+1}\}\cap C=\emptyset$ if $F(j,k+1)=F(k+1,j)=j$ and there does not exist $i, 1\leq i\leq m$ such that $j,k+1\in \{i_1,i_1+1,\cdots, i_1+t_i+1\}$.
\end{enumerate}
\end{enumerate}
\end{theorem}

\begin{example}
Let $\Phi$ be of type $A_2$. In Table \ref{parabolictable} we list all biclosed sets in $\Phi^+_{\emptyset,\{\alpha_2\}}=\{\alpha_1,\alpha_1+\alpha_2,\alpha_2,-\alpha_2\}$, their corresponding biclosed sets in $\Phi$ and corresponding quasitrivial semigroup structures.

\begin{table}
\begin{tabular}{|c|c|c|}
  \hline
  Biclosed set in $\Phi^+_{\emptyset,\{\alpha_2\}}$  & Biclosed set in $\Phi$ & semigroup structure\\
  \hline
  $\emptyset$ & $\{-\alpha_1,-\alpha_1-\alpha_2\}$ & $\{2,3\}^1\prec 1$ \\
  $\{\alpha_2\}$ & $\{\alpha_2,-\alpha_1,-\alpha_1-\alpha_2\}$ & $2\prec 3\prec 1$ \\
  $\{-\alpha_2\}$ & $\{-\alpha_2,-\alpha_1,-\alpha_1-\alpha_2\}$ & $3\prec 2\prec 1$ \\
  $\{\alpha_2,-\alpha_2\}$ & $\{\pm\alpha_2,-\alpha_1,-\alpha_1-\alpha_2\}$ & $\{2,3\}^2\prec 1$ \\
  $\{\alpha_2,\alpha_1+\alpha_2\}$ & $\{\alpha_2,\alpha_1+\alpha_2,-\alpha_1\}$ & $2\prec 1\prec 3$ \\
  $\{-\alpha_2,\alpha_1\}$ & $\{-\alpha_2,\alpha_1,-\alpha_1-\alpha_2\}$ & $3\prec 1\prec 2$ \\
  $\{\alpha_1+\alpha_2,\alpha_1\}$ & $\{\alpha_1+\alpha_2,\alpha_1\}$ & $1\prec \{2,3\}^1$ \\
 $\{\alpha_1,\alpha_2,\alpha_1+\alpha_2\}$ & $\{\alpha_1,\alpha_2,\alpha_1+\alpha_2\}$ & $1\prec 2\prec 3$ \\
$\{-\alpha_2,\alpha_1,\alpha_1+\alpha_2\}$ & $\{-\alpha_2,\alpha_1,\alpha_1+\alpha_2\}$ & $1\prec 3\prec 2$ \\
$\{\pm\alpha_2,\alpha_1,\alpha_1+\alpha_2\}$ & $\{\pm\alpha_2,\alpha_1,\alpha_1+\alpha_2\}$ & $1\prec \{2,3\}^2$ \\
\hline
\end{tabular}
\caption{Biclosed sets in $\Phi^+_{\emptyset,\{\alpha_2\}}$    and corresponding quasitrivial semigroup structures}\label{parabolictable}
\end{table}
\end{example} 

Combining Lemma \ref{parabolic} and theorem \ref{bjpara}, we can immediately arrive at the following enumeration result.
Let $\{A_1, A_2, \dots A_t\}$ and $\{A_1', A_2', \cdots, A_p'\}$ be two ordered partition of the set $\{1,2,\dots, n\}$.
We write $\{A_1', A_2', \cdots, A_p'\}\unlhd\{A_1, A_2, \dots A_t\}$ if each $A_i$ is a union of $A_{j}', A_{j+1}', \dots, A_k'$ (i,e, the former is a refinement of the latter).  Let $\Phi^+_{\emptyset, \Delta_2}$ be a parabolic set and let $F$ be the corresponding quasitrivial semigroup structure. Denote the associated total preorder by $\prec_F$. Such a total preorder naturally induces a partition of $\{1,2,\dots, n+1\}$, which we will denote by $P$.
We denote by $b(P)$ the number of subsets in $P$ whose size is greater than 1.

\begin{theorem}\label{countbiclosedset}
Let $\Phi$ be of type $A_n$. The number of biclosed sets in the  standard parabolic set $\Phi^+_{\emptyset, \Delta_2}$ is given by
$$\sum_{P'\unlhd P}(|P'|!)2^{b(P')}.$$
\end{theorem}

\begin{example}
Consider the biclosed sets in rank 1 parabolic set $\Phi^+_{\emptyset,\{\alpha_i\}}$. These are in bijection with the total orders on 
$\{1, 2, \dots, n+1\}$ and total orders on $\{1, 2, \dots, i-1, (i, i+1), i+2, \dots, n+1\}$. In this case $$P=\{\{1\}, \{2\}, \dots, \{i-1\}, \{i, i+1\}, \{i+2\}, \dots, \{n+1\}\},$$
$$P'=\{\{1\}, \{2\}, \dots, \{n+1\}\}.$$
Then $b(P)=1$ and $b(P')=0$. Then by Theorem \ref{countbiclosedset}, the number of biclosed sets is 
$$2(n!)+(n+1)!.$$
This sequence is essentially OEIS sequence A052572.
\end{example}

\section{Associativity index of quasitrivial magmas}\label{indexsec}

Let $C$ be an arbitrary subset of $\Phi$ (of type $A_n$). Then we have a quasitrivial magma structure $F_C$ on the set $\{1,2,\cdots, n+1\}$ as defined in Section \ref{mainbijA}.
The \emph{associativity index} of $F_C$ is defined as the cardinality of the set
$$\{(a,b,c)|\,F_C(F_C(a,b),c)\neq F_C(a,F_C(b,c))\}.$$

For such a (related) notion, see for example \cite{assoindex} (where the associativity index is instead defined as the number of triples that satisfy the associativity law).

Define the \emph{index} of $C$ as the cardinality of the set
$$\{\{\alpha,\beta\},\alpha\neq \beta,\alpha,\beta\in C|\,\alpha+\beta\not\in C\}\cup \{\{\alpha,\beta\},\alpha\neq \beta,\alpha,\beta\not\in C|\,\alpha+\beta\in C\}.$$
One sees readily that the biclosed sets are precisely subsets of $\Phi$ having index 0. Therefore this index can be viewed as a measure of the deviation of $C$ from biclosedness. 

We aim to prove
\begin{theorem}\label{indexthm}
The index of $C$ is equal to the associativity index of $F_C.$
\end{theorem}

This theorem can be seen as a generalization of our  Theorem \ref{mainbijection}.

Let $\alpha,\beta,\gamma$ be three different roots. We say that a subset $C$ of $\Phi$ separates $\{\alpha,\beta\}$ and $\{\gamma\}$ if $\alpha+\beta=\gamma$ and $C\cap\{\alpha,\beta,\gamma\}=\{\alpha,\beta\}$ or $\{\gamma\}$.

\begin{lemma}\label{atwotriple}
Let $\Phi$ be of type $A_2.$
A subset $C$ of $\Phi$ separates $\{\alpha_{i,j},\alpha_{j,k}\}$ and $\{\alpha_{i,k}\}$ if and only if $F_C(F_C(i,j),k)\neq F_C(i,F_C(j,k))$.
\end{lemma}

\begin{proof}
First we consider the case where $i=1,j=2$ and $k=3.$

There are $2^6-20=44$ non-biclosed subsets of $\Phi$ listed below.
$$C_1=\{\alpha_{1,2}\},\,\,\, C_2=\{\alpha_{2,1}\},\,\,\, C_3=\{\alpha_{1,3}\},\,\,\, C_4=\{\alpha_{3,1}\},\,\,\,C_5=\{\alpha_{2,3}\},$$
$$C_6=\{\alpha_{3,2}\},\,\,\,C_7=\{\alpha_{1,2},\alpha_{2,3}\},\,\,\,C_8=\{\alpha_{1,2},\alpha_{2,1}\},\,\,\,C_9=\{\alpha_{1,2},\alpha_{3,1}\},$$
$$C_{10}=\{\alpha_{2,3},\alpha_{3,2}\},\,\,\,C_{11}=\{\alpha_{2,3},\alpha_{3,1}\},\,\,\,C_{12}=\{\alpha_{1,3},\alpha_{3,1}\},\,\,\,C_{13}=\{\alpha_{1,3},\alpha_{2,1}\},$$
$$C_{14}=\{\alpha_{3,2},\alpha_{2,1}\},\,\,\,C_{15}=\{\alpha_{3,2},\alpha_{1,3}\},\,\,\,C_{16}=\{\alpha_{1,2},\alpha_{2,3},\alpha_{3,1}\},$$
$$C_{17}=\{\alpha_{1,2},\alpha_{2,1},\alpha_{2,3}\},\,\,\,C_{18}=\{\alpha_{1,2},\alpha_{2,1},\alpha_{1,3}\},\,\,\,C_{19}=\{\alpha_{1,2},\alpha_{2,1},\alpha_{3,2}\},$$
$$C_{20}=\{\alpha_{1,2},\alpha_{2,1},\alpha_{3,1}\},\,\,\,C_{21}=\{\alpha_{2,3},\alpha_{3,2},\alpha_{1,2}\},\,\,\,C_{22}=\{\alpha_{2,3},\alpha_{3,2},\alpha_{2,1}\}$$$$C_i=\Phi\backslash C_{i-22}, 23\leq i\leq 44.$$
Write $J=\{3,7,12,13,15,16,17,21\}$.
One readily checks that $C_j, j\in J$ and $C_{j+22}, j\in J$ are precisely subsets of $\Phi$ which separates $\{\alpha_{1,2},\alpha_{2,3}\}$ and $\{\alpha_{1,3}\}$.
We calculate that
$$F_{C_j}(F_{C_j}(1,2),3)=3\neq1=F_{C_j}(1,F_{C_j}(2,3)), j\in J,$$
$$F_{C_{j+22}}(F_{C_{j+22}}(1,2),3)=3\neq1=F_{C_{j+22}}(1,F_{C_{j+22}}(2,3)), j\in J.$$
For $j\in \{1,2,\cdots,44\}\backslash (J\cup \{j+22|\,j\in J\})$, one can check that
$$F_{C_j}(F_{C_j}(1,2),3)=F_{C_j}(1,F_{C_j}(2,3)).$$
More precise, for $j\in \{2,4,5,6,10,11,14,22,40\},$
$$F_{C_j}(F_{C_j}(1,2),3)=1=F_{C_j}(1,F_{C_j}(2,3)).$$
For $j\in \{1,8,9,18,19,20,27,32,33,44\},$
$$F_{C_j}(F_{C_j}(1,2),3)=2=F_{C_j}(1,F_{C_j}(2,3)).$$
For $j\in \{23,24,26,28,30,31,36,41,42\},$
$$F_{C_j}(F_{C_j}(1,2),3)=3=F_{C_j}(1,F_{C_j}(2,3)).$$
By Theorem \ref{mainbijection}, for a biclosed set $C$ in $\Phi$, $F_C(F_C(1,2),3)=F_C(1,F_C(2,3))$. Therefore we have proved the lemma in the case where $i=1,j=2$ and $k=3.$
One can also define a $S_3$-action on the set of quasitrivial magma structures by $(\sigma F)(i,j)=\sigma(F(\sigma^{-1}(i),\sigma^{-1}(j))).$
Take a permutation $\sigma$ such that $\sigma(1)=i, \sigma(2)=j$ and $\sigma(3)=k.$ Then
$$(\sigma(F_C))((\sigma(F_C))(i,j),k)=(\sigma(F_C))(i,(\sigma(F_C))(j,k))$$
$$\Leftrightarrow F_C(F_C(\sigma^{-1}(i),\sigma^{-1}(j)),\sigma^{-1}(k))=F_C(\sigma^{-1}(i),F_C(\sigma^{-1}(j),\sigma^{-1}(k))).$$
From the definition of $F_C$, one easily sees that $\sigma F_C=F_{\sigma(C)}$. Combining the above observations, we see that
$$F_{\sigma(C)}(F_{\sigma(C)}(i,j),k)=F_{\sigma(C)}(i,F_{\sigma(C)}(j,k))$$
$$\Leftrightarrow (\sigma(F_C))((\sigma(F_C))(i,j),k)=(\sigma(F_C))(i,(\sigma(F_C))(j,k))$$
$$\Leftrightarrow F_C(F_C(\sigma^{-1}(i),\sigma^{-1}(j)),\sigma^{-1}(k))=F_C(\sigma^{-1}(i),F_C(\sigma^{-1}(j),\sigma^{-1}(k)))$$
$\Leftrightarrow C$ separates $\{\alpha_{1,2},\alpha_{2,3}\}$ and $\{\alpha_{1,3}\}$, which is equivalent to the  condition that $\sigma(C)$ separates $\{\alpha_{\sigma(1),\sigma(2)}=\alpha_{i,j}, \alpha_{\sigma(2),\sigma(1)}=\alpha_{j,i}\}$ and $\{\alpha_{\sigma(1),\sigma(3)}=\alpha_{i,j}\}$.
\end{proof}

\begin{lemma}\label{separation}
Let $\Phi$ be of type $A_n.$
A subset $C$ of $\Phi$ separates $\{\alpha_{i,j},\alpha_{j,k}\}$ and $\{\alpha_{i,k}\}$ if and only if $F_C(F_C(i,j),k)\neq F_C(i,F_C(j,k))$.
\end{lemma}

\begin{proof}
The roots $\alpha_{i,j},\alpha_{j,k},\alpha_{i,k},\alpha_{j,i},\alpha_{k,j}$ and $\alpha_{k,i}$ form a root (sub)system $\Phi'$ of type $A_2.$
By definition, a subset $C$ of $\Phi$ separates $\{\alpha_{i,j},\alpha_{j,k}\}$ and $\{\alpha_{i,k}\}$ if and only if $C\cap\Phi'$ separates $\{\alpha_{i,j},\alpha_{j,k}\}$ and $\{\alpha_{i,k}\}$. By Lemma \ref{atwotriple}, the latter is equivalent to $F_{C\cap \Phi'}(F_{C\cap \Phi'}(i,j),k)\neq F_{C\cap \Phi'}(i,F_{C\cap \Phi'}(j,k))$ where $F_{C\cap \Phi'}$ is the quasitrivial magma structure on $\{i,j,k\}$ induced by $C\cap \Phi'$. Now the lemma follows from the fact that when restricted to $\{i,j,k\}$, $F_C$ coincides with $F_{C\cap \Phi'}$.
\end{proof}

The following lemma is easily verifiable.

\begin{lemma}\label{twoequal}
Let $F$ be a quasitrivial magma structure on $\{1,2,\cdots, n\}$. Then one has
$F(F(i,j),i)=F(i,F(j,i)), F(F(i,i),j)=F(i,F(i,j))$ and $F(F(j,i),i)=F(j,F(i,i))$.
\end{lemma}

Now Theorem \ref{indexthm} is the consequence of Lemma \ref{separation} and Lemma \ref{twoequal}.
\begin{example}
Let $\Phi$ be of type $A_2$. The subset $C=\{\alpha_{1,2},\alpha_{2,3}\}$ of $\Phi$ has index 3 since it separates the following three pairs of sets
$$(1)\,\, \{\alpha_{1,2},\alpha_{2,3}\}\,\,\text{and}\,\,\{\alpha_{1,3}\},$$
$$(2)\,\, \{\alpha_{2,1},\alpha_{1,3}\}\,\,\text{and}\,\,\{\alpha_{2,3}\},$$
$$(3)\,\, \{\alpha_{3,2},\alpha_{1,3}\}\,\,\text{and}\,\,\{\alpha_{1,2}\}.$$
Correspondingly $$\{(a,b,c)|\,F_C(F_C(a,b),c)\neq F_C(a,F_C(b,c))\}=\{(1,3,2),(2,1,3),(1,2,3)\}.$$
\end{example}

\begin{remark}
Since the index of $C\subset \Phi$ is a measure of the deviation of $C$ from the biclosedness. One can ask what are the possible values of the indices of $C$. A particular interest might be those $C$ with maximal index (i.e. the ``least" biclosed set). Enumerating them is an interesting problem.
\end{remark}

\section{quasitrivial semigroup of type $B_n$}\label{bnsec}

Inspired by the bijection between biclosed sets in $\Phi$ of type $A_n$ and the quasitrivial semigroup structures on $\{1,2,\dots, n+1\}$. We define a quasitrivial semigroup of type $B_n$.

A total preorder on the set $\{\pm 1, \pm 2, \cdots, \pm n\}$ is \emph{admissible} if the induced equivalence relation satisfies the following properties:

(1) For an equivalence class $B$, either $B=-B$ or $B\cap -B=\emptyset$;

(2) If $B$ is an equivalence class, then so is $-B$;

(3) Let $A,B$ be two equivalence classes. If $A\prec B$, then $-B\prec -A.$

Condition (3) forces that there can be at most one equivalence class $B$ with the property $B=-B.$

We call an admissible total preorder on $\{\pm 1, \pm 2, \cdots, \pm n\}$ \emph{standard} if $-n\precsim -n-1\precsim \cdots\precsim -1\precsim 1\precsim 2\precsim \cdots \precsim n.$

A quasitrivial semigroup structure on the set $\{\pm 1, \pm 2, \cdots, \pm n\}$ is said to be of $B_n$ if and only if

(1) the associated total preorder is admissible and

(2) for an equivalence class $B$, $F|_{B}=F|_{-B}=\pi_i, i\in \{1,2\}$.

Therefore quasitrivial semigroups of type $B_n$ are in bijection with admissible total preorders on the set $\{\pm 1, \pm 2, \cdots, \pm n\}$
 whose induced equivalence classes are labelled with $i, i\in \{1,2\}$ and the labels of the equivalence classes $B$ and $-B$ are the same.


One can see immediately that the group $S_n^B$ acts on the set of admissible total preorders (and quasitrivial semigroup structures of type $B$ ).

Now let the root system $\Phi$ be of type $B_n$. For $a,b\in \{1,2,\dots,n\}, a\leq b$, write 
$$\alpha_{-a,b}=(\alpha_1+\alpha_2+\dots+\alpha_a)+(\alpha_1+\alpha_2+\dots+\alpha_{b}),$$ 
$$=2\alpha_1+2\alpha_2+\dots+2\alpha_a+\alpha_{a+1}+\dots+\alpha_b, a\neq b$$
$$\alpha_{-a,a}=\alpha_1+\alpha_2+\dots+\alpha_a,$$
$$\alpha_{a,b}=\alpha_{a+1}+\alpha_{a+2}+\dots+\alpha_b, a<b.$$
Write $\alpha_{n,m}:=-\alpha_{m,n}, \alpha_{-n,-m}=-\alpha_{n,m}$. One may check that $w\alpha_{n,m}=\alpha_{w(n),w(m)}.$

The main result of this section is the following.

\begin{theorem}
\begin{enumerate}
\item The biclosed sets in a root system $\Phi$ of type $B_n$  are in bijection with the  quasi-trivial semigroup structures of type $B_n$.
\item Let $C$ be a biclosed set in $\Phi$ and $F_C$ be the corresponding quasitrivial associative binary operation. Assume that $t<m.$
\begin{enumerate}
\item $\{\alpha_{t,m},-\alpha_{t,m}\}\cap C=\emptyset$ if and only if $F_{C}(t,m)=t, F_{C}(m,t)=m.$
\item $\alpha_{t,m}\in C$ and $-\alpha_{t,m}\not\in C$ if and only if  $F_{C}(t,m)=F_{C}(m,t)=m.$
\item $-\alpha_{t,m}\in C$ and $\alpha_{t,m}\not\in C$ if and only if $F_{C}(t,m)=F_{C}(m,t)=t.$
\item $\{\alpha_{t,m},-\alpha_{t,m}\}\subset C$ if and only if $F_{C}(t,m)=m, F_{C}(m,t)=t.$
\end{enumerate}
\end{enumerate}
\end{theorem}

\begin{proof}
(1) First we show that
the quasitrivial semigroups of type $B_n$, whose induced total preorder is standard are in bijection with the biclosed sets of the form $\Phi_{\Delta_1,\Delta_2}^+$ where $\Phi^+$ is the standard positive system.

From the symmetry of the allowed total preorder on $\{\pm 1, \pm 2, \cdots, \pm n\}$, one sees that the standard admissible total preorders on $\{\pm 1, \pm 2, \cdots, \pm n\}$ are in bijection with their restrictions to $\{\pm 1, 2,\cdots, n\}$. We will from now on  call these restrictions the admissible total preorders on $\{\pm 1, 2,\cdots, n\}$.
As a result type $B_n$ quasitrivial semigroups whose induced total preorder is standard are in bijection with  admissible total preorders on $\{\pm 1, 2,\cdots, n\}$ with induced equivalence classes labelled by $1$ or $2.$

Write $b_1=-1, b_i=i-1, 2\leq i\leq n+1.$
Note that if $b_i\backsim b_j, i<j$ then $b_i\backsim b_{i+1}\backsim \dots \backsim b_j$. This follows immediately from the fact that the induced total preorder is admissible and the assumption that our quasitrivial semigroup structures have a standard induced total preorder.
Therefore a typical admissible  total preorder on $\{\pm 1, 2,\cdots, n\}$ takes the form:
$$b_1\prec b_2\prec \cdots b_{i_1-1}\prec \{b_{i_1}, b_{i_1+1}, \cdots, b_{i_1+k_1}\}$$
$$\prec b_{i_1+k_1+1}\prec b_{i_1+k_1+2}\prec \cdots b_{i_2-1}\prec \{b_{i_2}, b_{i_2+1}, \cdots, b_{i_2+k_2}\}$$
$$\cdots$$
$$\prec b_{i_{t-1}+k_{t-1}+1}\prec b_{i_{t-1}+k_{t-1}+2}\prec \cdots b_{i_t-1}\prec \{b_{i_t}, b_{i_t+1}, \cdots, b_{i_t+k_t}\}$$
$$\prec b_{i_{t}+k_{t}+1}\prec b_{i_{t}+k_{t}+2}\prec \cdots b_{n+1}.$$
By Lemma \ref{bijection1}, these preorders are in bijective with  vectors  of positive integers
$$(i_1,i_2,\cdots,i_t):1\leq i_1, i_t+k_t-1\leq n, i_j+k_j+1\leq i_{j+1} \,\,\text{for}\,\, 1\leq j\leq t-1.$$
Such vectors are further in bijection with a collection of (pairwise orthogonal) subsets of $\Delta$:
$$\{\Delta_1,\Delta_2,\cdots,\Delta_t: \Delta_j=\{\alpha_{i_j},\alpha_{i_j+1},\cdots,\alpha_{i_j+k_j-1}\}\}.$$
Therefore we can conclude that admissible total preorders on $\{\pm 1, 2,\cdots, n\}$ with induced equivalence classes labelled by $1$ or $2$ are in bijection with
these collections of subsets with each $\Delta_i$ labelled with $1$ or $2$.
Let $\Delta_1'$ ($\Delta_2'$) be the union of $\Delta_i$ which is labelled $1$ (resp. $2$).
Hence admissible total preorders on $\{\pm 1, 2,\cdots, n\}$ with induced equivalence classes labelled by $1$ or $2$ are in bijection with the biclosed sets
of the form $\Phi^+_{\Delta_1',\Delta_2'}$.

The stabilizer of a standard admissible total preorder on $\{\pm 1, \pm 2, \cdots, \pm n\}$ is the standard parabolic subgroup generated by the simple reflections $(i,i+1)(-i,-i-1)$ when $i$ and $i+1$  are equivalent under this preorder and the simple reflection $(-1,1)$ if $-1$ and $1$ are equivalent under this preorder. This stabilizer coincides with the stabilizer of the biclosed set $\Phi^+_{\Delta_1',\Delta_2'}$ corresponding to this total preorder.
Therefore the $W$-equivariant bijection follows.

(2) By a routine check, one can verify that for $1\leq t<m\leq n$ and $-n\leq a,b\leq n$, $$\alpha_{a,b}\in \Phi_{\{\alpha_{t+1},\alpha_{t+2},\dots,\alpha_m\}}$$
if and only if $t\leq a<b\leq m$ or $-m\leq -b<-a\leq t$ and for $1\leq m\leq n$ and $-n\leq a,b\leq n$, $$\alpha_{a,b}\in \Phi_{\{\alpha_{1},\alpha_{2},\dots,\alpha_m\}}$$
if and only if $1\leq |a|, |b| \leq m.$ This proves the case where the associated admissible total preorder is standard. 

Then (2) can be proved by applying the proof of Lemma \ref{twoelt} almost verbatim to this $B_n$ situation thanks to the fact that $w\alpha_{n,m}=\alpha_{w(n),w(m)}$.
\end{proof}

\begin{example}
In this example  the bijection between the biclosed sets in the root system of $B_2$ and quasitrivial semigroup structures on $\{\pm1,\pm2\}$ is presented in detail. (See Table \ref{typebtable}.)
For an equivalence class $A\subset \{\pm1,\pm2\}$, we write $A^i, i=1,2$ if the quasitrivial semigroup structure restricted to $A$ is $\pi_i$. Denote the short simple root by $\alpha_1$ and the long simple root by $\alpha_2$.

\begin{table}\tiny
\begin{tabular}{|c|c|}
  \hline
  Biclosed set & Quasitrivial semigroup structure \\
    \hline
  $\Phi^+_{\{\alpha_1,\alpha_2\},\emptyset}=\emptyset$ & $\{-2,-1,1,2\}^1$ \\
  $\Phi^+_{\{\alpha_1\},\emptyset}=\{\alpha_1+\alpha_2,\alpha_2,2\alpha_1+\alpha_2\}$ & $-2\prec\{-1,1\}^1\prec 2$\\
  $\Phi^+_{\{\alpha_2\},\emptyset}=\{\alpha_1+\alpha_2,\alpha_1,2\alpha_1+\alpha_2\}$ & $\{-2,-1\}^1\prec\{1,2\}^1$\\
  $s_2\Phi^+_{\{\alpha_1\},\emptyset}=\{\alpha_1,-\alpha_2,2\alpha_1+\alpha_2\}$ & $-1\prec\{-2,2\}^1\prec 1$\\
  $s_1s_2\Phi^+_{\{\alpha_1\},\emptyset}=\{-\alpha_1,\alpha_2,-2\alpha_1-\alpha_2\}$ & $1\prec\{-2,2\}^1\prec -1$\\
  $s_2s_1s_2\Phi^+_{\{\alpha_1\},\emptyset}=\{-\alpha_1-\alpha_2,-\alpha_2,-2\alpha_1-\alpha_2\}$ & $2\prec\{-1,1\}^1\prec -2$\\

  $s_1\Phi^+_{\{\alpha_2\},\emptyset}=\{\alpha_2,-\alpha_1,\alpha_1+\alpha_2\}$ & $\{-2,1\}^1\prec\{-1,2\}^1$\\
  $s_2s_1\Phi^+_{\{\alpha_2\},\emptyset}=\{-\alpha_1-\alpha_2,-\alpha_2,\alpha_1\}$ & $\{-1,2\}^1\prec\{-2,1\}^1$\\
  $s_1s_2s_1\Phi^+_{\{\alpha_2\},\emptyset}=\{-2\alpha_1-\alpha_2,-\alpha_1-\alpha_2,-\alpha_1\}$ & $\{1,2\}^1\prec\{-2,-1\}^1$\\

  $\Phi^+_{\emptyset,\emptyset}=\{\alpha_1,\alpha_1+\alpha_2,\alpha_2,2\alpha_1+\alpha_2\}$ & $-2\prec -1\prec 1\prec 2$\\
  $s_1\Phi^+_{\emptyset,\emptyset}=\{-\alpha_1,\alpha_1+\alpha_2,\alpha_2,2\alpha_1+\alpha_2\}$ & $-2\prec 1\prec -1\prec 2$\\
  $s_2s_1\Phi^+_{\emptyset,\emptyset}=\{\alpha_1,-\alpha_1-\alpha_2,-\alpha_2,2\alpha_1+\alpha_2\}$ & $-1\prec 2\prec -2\prec 1$\\
   $s_1s_2s_1\Phi^+_{\emptyset,\emptyset}=\{-\alpha_1,-\alpha_1-\alpha_2,\alpha_2,-2\alpha_1-\alpha_2\}$ & $1\prec 2\prec -2\prec -1$\\
  $s_2s_1s_2s_1\Phi^+_{\emptyset,\emptyset}=\{-\alpha_1,-\alpha_1-\alpha_2,-\alpha_2,-2\alpha_1-\alpha_2\}$ & $2\prec 1\prec -1\prec -2$\\
  $s_2\Phi^+_{\emptyset,\emptyset}=\{\alpha_1,\alpha_1+\alpha_2,-\alpha_2,2\alpha_1+\alpha_2\}$ & $-1\prec -2\prec 2\prec 1$\\
$s_1s_2\Phi^+_{\emptyset,\emptyset}=\{\alpha_2,\alpha_1+\alpha_2,-\alpha_1,-2\alpha_1-\alpha_2\}$ & $1\prec -2\prec 2\prec -1$\\
$s_2s_1s_2\Phi^+_{\emptyset,\emptyset}=\{-\alpha_2,-\alpha_1-\alpha_2,\alpha_1,-2\alpha_1-\alpha_2\}$ & $2\prec -1\prec 1\prec -2$\\

$\Phi^+_{\emptyset,\{\alpha_1\}}=\{\pm\alpha_1,\alpha_1+\alpha_2,\alpha_2,2\alpha_1+\alpha_2\}$ & $-2\prec\{-1,1\}^2\prec 2$\\
  $\Phi^+_{\emptyset,\{\alpha_2\}}=\{\pm\alpha_2,\alpha_1+\alpha_2,\alpha_1,2\alpha_1+\alpha_2\}$ & $\{-2,-1\}^2\prec\{1,2\}^2$\\
  $s_2\Phi^+_{\emptyset,\{\alpha_1\}}=\{\alpha_1,-\alpha_2,2\alpha_1+\alpha_2,\pm(\alpha_1+\alpha_2)\}$ & $-1\prec\{-2,2\}^2\prec 1$\\
  $s_1s_2\Phi^+_{\emptyset,\{\alpha_1\}}=\{-\alpha_1,\alpha_2,-2\alpha_1-\alpha_2,\pm(\alpha_1+\alpha_2)\}$ & $1\prec\{-2,2\}^2\prec -1$\\
  $s_2s_1s_2\Phi^+_{\emptyset,\{\alpha_1\}}=\{-\alpha_1-\alpha_2,-\alpha_2,-2\alpha_1-\alpha_2,\pm\alpha_1\}$ & $2\prec\{-1,1\}^2\prec -2$\\

  $s_1\Phi^+_{\emptyset,\{\alpha_2\}}=\{\pm(2\alpha_1+\alpha_2),\alpha_2,-\alpha_1,\alpha_1+\alpha_2\}$ & $\{-2,1\}^2\prec\{-1,2\}^2$\\
  $s_2s_1\Phi^+_{\emptyset,\{\alpha_2\}}=\{\pm(2\alpha_1+\alpha_2),-\alpha_1-\alpha_2,-\alpha_2,\alpha_1\}$ & $\{-1,2\}^1\prec\{-2,1\}^1$\\
  $s_1s_2s_1\Phi^+_{\emptyset,\{\alpha_2\}}=\{\pm\alpha_2,-2\alpha_1-\alpha_2,-\alpha_1-\alpha_2,-\alpha_1\}$ & $\{1,2\}^2\prec\{-2,-1\}^2$\\
$\Phi^+_{\emptyset,\{\alpha_1,\alpha_2\}}=\emptyset$ & $\{-2,-1,1,2\}^2$ \\

\hline
\end{tabular}
\caption{Biclosed sets in a root system of type $B_2$ and corresponding quasitrivial semigroup structures}\label{typebtable}
\end{table}
\end{example}

\section{Oriented matroid structure from total preorder}\label{omsec}

A (reduced) oriented matroid can be defined as an involuted set together with a closure operator (or equivalently, defined as a signed set with a closure operator).
Here we follow \cite{largeconvex} and define
an \emph{oriented matroid} as a triple system $(E,-,c)$  where $E$ is a set with an involution map $-: E\rightarrow E$ (i.e. $(-)^2$ is the identity map) and a (finitary) closure operator $c: \mathcal{P}(E)\rightarrow \mathcal{P}(E)$  such that

   (1) $-c(X)=c(-X)$  for all $X\subseteq E$,
(2) if $X\subseteq E$ and  $x\in c(X\cup \{-x\})$, then $x\in c(X)$,
(3)
if $X\subseteq E$ and  $x,y\in E$ with  $x\in c(X\cup \{-y\})$ but $x\not\in c(X)$, then $y\in c((X\backslash\{y\})\cup\{-x\})$.
(4) $c(\emptyset)=\emptyset.$
 
 Here the closure operator $c$ is finitary, meaning that if $x\in c(X)$ then there exists a finite $Y\subset X$ such that $x\in c(Y).$ 
 
There are many equivalent ways to define an oriented matroid. One of them is to characterize an oriented matroid by its \emph{covectors}. 
A finite oriented matroid can be defined as a triple $(E,-,\mathcal{L})$ where $E$ is a finite set with an involution map $-: E\rightarrow E$ and $\mathcal{L} \subset \mathcal{P}(E)$ with the following properties:

(C1) $X\in \mathcal{L}$ implies $X\cap -X=\emptyset$;

(C2) $\emptyset \in \mathcal{L}$;

(C3) $X\in \mathcal{L}$ implies $-X\subset \mathcal{L}$;

(C4) $X, Y\in \mathcal{L}$ implies $X\circ Y\in \mathcal{L}$ where $X\circ Y$ is defined as follows:

(a) if $x\in X$, then $x\in X\circ Y$ and $-x\not\in X\circ Y$;

(b) if $x\not\in X, -x\not\in X$, then $(X\circ Y)\cap \{\pm x\}=Y\cap \{\pm x\}$;

(C5) if $X, Y\in \mathcal{L}$, $x\in X$ and $-x\in Y$, there exists $Z\in \mathcal{L}$ such that $\{\pm x\}\cap Z=\emptyset$ and   

(a) for all $y\in X\cap Y$, one has $y\in Z$ and $-y\not\in Z$;

(b) if $\{\pm y\}\cap X=\emptyset, \{\pm y\}\cap Y=\{\epsilon y\}$ ($\epsilon\in \{\pm\}$), $\epsilon y\in Z, -\epsilon y\not\in Z$;

(c) if $\{\pm y\}\cap Y=\emptyset, \{\pm y\}\cap X=\{\epsilon y\}$ ($\epsilon\in \{\pm\}$), $\epsilon y\in Z, -\epsilon y\not\in Z$;

(d) if $\{\pm y\}\cap X=\emptyset, \{\pm y\}\cap Y=\emptyset,$ then $\{\pm y\}\cap Z=\emptyset$.

An element of the collection $\mathcal{L}$ is called a covector of the oriented matroid $(E,-,\mathcal{L})$.
Given a root system $\Phi$ realized in the vector space $V$, it is known that the triple $(E,-,c)$ is an oriented matroid with the closure operator $c$  defined as:
\begin{equation*}
c(X):=\{\sum_{i\in I}k_iv_i\mid v_i\in X, k_i\in \mathbb{R}_{\geq 0}, I\,\text{is a finite set.}\}
\end{equation*}

However the root system can be defined in a purely combinatorial way. A combinatorial root system can be identified as $T\times \{\pm1\}$ where $T$ is the set of the reflections in the associated Coxeter group. Therefore it is hoped that a combinatorial proof of the existence of the oriented matroid structure on the abstract root system can be obtained without relying on the linearly realized root system. The (purely algebraic and combinatorial) construction and theory in \cite{DyerWangGroupoid} guarantees the existence of such an oriented matroid structure. However a simpler purely combinatorial proof (than the theory in \cite{DyerWangGroupoid}) of the existence of the oriented matroid structure of the root system is desirable. Here we use the previous total preorder characterization of (a particular type of) biclosed sets to give a purely combinatorial proof.  

It is known that for a finite irreducible root system $\Phi$, the set of covectors of $(\Phi,-,c)$ is precisely the set of horocyclic subsets. 
 
Now let $\Phi$ be an irreducible root system of type $A_n$.

In view of Lemma \ref{twoelt} and Proposition \ref{variouscorr}, to obtain a purely combinatorial proof of the oriented matroid structure of $(\Phi,-,c)$ for the type $A_n$ root system, we only need to verify the covector axioms phrased in terms of total preorders.

Let $(X_n,\precsim_1),(X_n,\precsim_2)$ be two total preorders on $X_n$.
We define a relation $(X_n,\precsim_3)$ as follows:

If $x\precnsim_1 y$ then $x\precnsim_3 y$. If $x\sim_1 y$ and  $x\precnsim_2 y$ then $x\precnsim_3 y.$
If $x\sim_1 y$ and $x\sim_2 y$ then $x\sim_3 y.$

We will denote $\precsim_3$ by $\precsim_1\circ \precsim_2.$

\begin{lemma}\label{ac4}
$(X_n, \precsim_3)$ is a well defined total preorder.
\end{lemma}

\begin{proof}
Let $x,y\in X_n$, we first show that the equivalence relation $\sim_3$ is well defined.
If $x\sim_3 y$ and $y\sim_3 z$, then by definition $x\sim_1 y, x\sim_2y, y\sim_1 z, y\sim_2 z$.
So $x\sim_1 z$ and $x\sim_2 z$. Therefore $x\sim_3 z.$

Next we show  that if $x\sim_3 y$ and $x\precnsim_3 z$ then $y\precnsim_3 z$. This condition implies two possibilities:

(1) $x\sim_1 y, x\sim_2y$ and $x\precnsim_1 z$. Therefore $y\precnsim_1 z$ and consequently $y\precnsim_3 z$.

(2) $x\sim_1 y, x\sim_2y, x\sim_1 z$ and $x\precnsim_2 z$. Therefore $y\sim_1 z$ and $y\precnsim_2 z$. Hence $y\precnsim_3 z.$

Similarly we can show that if $x\sim_3 y$ and $z\precnsim_3 x$ then $z\precnsim_3 y$.

Finally we show that if $x\precnsim_3 y$ and $y\precnsim_3 z$ then $x\precnsim_3 z.$

Case I. $x\precnsim_1 y$ and $y\precnsim_1 z$. So $x\precnsim_1 z$ and consequently $x\precnsim_3 z$.

Case II. $x\sim_1 y, x\precnsim_2 y$ and $y\precnsim_1 z$. So $x\precnsim_1 z$ and consequently $x\precnsim_3 z$.

Case III. $x\precnsim_1 y, y\sim_1 z$ and $y\precnsim_2 z$. So $x\precnsim_1 z$ and consequently $x\precnsim_3 z$.

Case IV. $x\sim_1 y, x\precnsim_2y, y\sim_1 z$ and $y\precnsim_2 z$. So $x\sim_1 z$ and $x\precnsim_2 z$. Consequently $x\precnsim_3 z.$
\end{proof}

Let $\precsim_1$ and $\precsim_2$ be two total preorders on $\{1,2,\dots,n\}$. Assume that $x\precnsim_1 y$ and $y\precnsim_2 x$. Denote by $\precsim_3=\precsim_1\circ\precsim_2$. Let $X$ be the equivalent class containing $x$ and let $Y$ be the equivalent class containing $y$ under $\precsim_3$. Assume that $X\precnsim_3X_1\precnsim_3 X_2\precnsim_3\cdots \precnsim_3X_p\precnsim_3 Y$. Denote $L:=X\cup Y$. 
Now we define a partial order on $L, X_1, X_2, \dots, X_p.$

First, we note that $z_1\sim_3 z_2$ implies that $z_1\sim_1 z_2$ and $z_1\sim_2 z_2.$ Therefore all elements in $X_i$ remain equivalent under $\precsim_1$ and 
$\precsim_2$. Now we define the partial order as follows:

(1) If all elements in $X_i$ are equivalent to $Y$ under $\precsim_2$ or $\precsim_1$, then $X_i<L$.

(2) If all elements in $X_i$ are equivalent to $X$ under $\precsim_2$ or $\precsim_1$, then $L<X_i$.

(3) If $X_i\precnsim_1 Y$ and  $X_i\precnsim_2 Y$, then $X_i<L$

(4) If $X\precnsim_1 X_i$ and  $X\precnsim_2 X_i$, then $L<X_i$

 We claim that it is impossible for both $L<X_i$ and $X_i<L$ to hold simultaneously.
 
 \emph{Proof of the claim}: 
 First we show that the conditions (1) and (2) cannot simultaneously hold.
 
  Since $X\precnsim_1 Y$ and $Y\precnsim_2 X$, it cannot happen that $X\sim_1 X_i\sim_1 Y$ or $X\sim_2 X_i\sim_2 Y$.
Suppose that $X\sim_1 X_i$ and $Y\sim_2 X_i$. Since $X\precnsim_3 X_i$,  $X\precnsim_2 X_i$. On the other hand, $Y\precnsim_2 X$ implies $X_i\precnsim X$.
A contradiction. Similar argument shows that $Y\sim_1 X_i$ and $X\sim_2 X_i$ cannot simultaneously hold.

Second, we show that the conditions (1) and (4) cannot simultaneously hold.
 
Assume to the contrary that they both hold. Suppose that $X_i\sim_1 Y$. Since $X_i\precnsim_3 Y$, $X_i\precnsim_2 Y$. This implies $X\precnsim_2 X_i\precnsim Y$, contradicting $Y\precnsim_2 X.$  Suppose that $X_i\sim_2 Y$. Since $Y\precnsim_2 X$, $X_i\precnsim_2 X$, contradicting $X\precnsim_2 X_i.$

Similarly, we can prove that the conditions (2) and (3), conditions (3) and (4) cannot simultaneously hold.  

(5) If $X_i\precnsim_1 X_j$ and $X_i\precnsim_2 X_j$, then $X_i<X_j$.

(6) If $X_i\precnsim_1 X_j$ and $X_i\sim_2 X_j$, then $X_i<X_j$.

(7) If $X_i\sim_1 X_j$ and $X_i\precnsim_2 X_j$, then $X_i<X_j$.

Then one takes the transitive closure.

\begin{lemma}
The above defined partial order $<$ is well-defined.
\end{lemma}

\begin{proof}
We verify that the relations defined above do not introduce circuits (this implies that their transitive closure is a well-defined partial order).

CASE I. Suppose that $X_{i_1}<X_{i_2}<\dots<X_{i_r}<L$. Then (1) $X_{i_1}\precsim_2 X_{i_r}$ and $X_{i_1}\precsim_1 X_{i_r}$.
(2) $X_{i_r}\sim_1 Y$ or $X_{i_r}\sim_2 Y$ or $X_{i_r}\precnsim_j Y, j=1,2$. Note that $X_{i_r}\sim_1 Y$ implies $X_{i_r}\precnsim_2 Y$ and that $X_{i_r}\sim_2 Y$ implies $X_{i_r}\precnsim_2 Y$. Therefore $X_{i_1}\precsim_1 Y$ and $X_{i_1}\precsim_2 Y.$
So   $X_{i_1}<L$. 

CASE II. Suppose that $L<X_{j_1}<X_{j_2}<\dots<X_{j_s}$. This case is dual to CASE I and can be proved similarly.

CASE III. Suppose that $X_{i_1}<X_{i_2}<\dots<X_{i_r}<L<X_{j_1}<X_{j_2}<\dots<X_{j_s}$. 
The previous arguments show that (1)  $X_{i_1}\precsim_2 Y$ and (2) $X\precsim_2 X_{j_s}$. Therefore $X_{i_1}\precsim_2 Y\precnsim_2 X\precsim X_{j_s}.$
Hence either $X_{i_1}<X_{j_s}$ or they have no relation before the transitive closure is introduced.

CASE IV. Suppose that $X_{i_1}<X_{i_2}<\dots<X_{i_r}$. Then by our construction one must have $X_{i_1}<X_{i_r}$.
\end{proof}

Then by the Szpilrajn extension theorem, we can extend the partial order $<$ to a total order on $L, X_1, X_2, \dots, X_p$.
We modify $\precsim_3$ by replacing the interval $[X, Y]$ by the above total order on $L, X_1, X_2, \dots, X_p$.
In this way we obtain a total preorder $\precsim_4$.

Our construction immediately yields

\begin{lemma}\label{ac5}
Under $\precsim_4$ one has

(1) if $z_1\precnsim_1 z_2$ and $z_1\precnsim_2 z_2$, then $z_1\precnsim_4 z_2;$

(2) if $z_1\sim_1 z_2$ and $z_1\precnsim_2 z_2$, then $z_1\precnsim_4 z_2;$

(3) if $z_1\sim_2 z_2$ and $z_1\precnsim_1 z_2$, then $z_1\precnsim_4 z_2.$

(4) if $z_1\sim_1 z_2$ and $z_1\sim_2 z_2$, then $z_1\sim_4 z_2.$

(5) $x\sim_4 y.$
\end{lemma}

\begin{theorem}
Let $\Phi$ be a root system of type $A_n$.  Then the family $L$ consisting of all  horocyclic sets of $\Phi$  satisfies the covector axioms of an oriented matroid.
\end{theorem}

\begin{proof}
By Lemma \ref{twoelt} and Proposition \ref{variouscorr} (4) horocyclic sets correspond to quasitrivial semigroup structures with the property that, when restricted to any two different elements, the multiplication cannot be equal to the projection map $\pi_2.$ Therefore they further correspond bijectively to total preorders on $\{1,2,\cdots,n+1\}$ thanks to the rules in Lemma  \ref{twoelt} and the standard association between total preorders and quasitrivial semigroup structures. Therefore, it is sufficient to show the axioms (C1)-(C5) for the total preorders.

(C1) follows from the rule of association and the fact that we tacitly assume that all equivalent classes are labelled 1 for the total preorder in question.
(C2) follows from the fact that a single equivalent class $\{1,2,\cdots, n+1\}$ is a (trivial) total preorder. 
(C3) follows from the fact that one can reverse a total preorder. (C4) follows from Lemma \ref{ac4} and (C5) follows from Lemma \ref{ac5}.
\end{proof}

\end{document}